\documentclass[12pt,twoside,english,a4paper]{article}
\usepackage{babel,amssymb,amsthm}
\usepackage{graphicx}
\usepackage{amsmath}
\usepackage{bm}

\newcommand{\punto}{\,\cdot\,}

\newcommand{\smallfrac}[2]{{\textstyle\frac{#1}{#2}}} 

\newtheorem{proposition}{Proposition}
\newtheorem{corollary}{Corollary}
\newtheorem{lemma}{Lemma}
\newtheorem{theorem}{Theorem}

\title{A Nystr\"om method for the two dimensional \\ Helmholtz hypersingular
equation}

\date{\today}

\author{V\'\i ctor Dom\'\i nguez\footnote{Departamento de Ingenier\'\i a
Matem\'atica e Inform\'atica, Universidad P\'ublica de Navarra, 31500 Tudela,
Spain. {\tt victor.dominguez@unavarra.es}. Partially supported by MICINN Project
MTM2010-21037}, Sijiang L. Lu\footnote{Department of Mathematical Sciences,
University of Delaware, USA. {\tt sjly@math.udel.edu}} \& Francisco--Javier
Sayas\footnote{Department of Mathematica Sciences, University of Delaware,
Newark, DE 19716, USA. {\tt fjsayas@math.udel.edu}. Partially supported by NSF
grant DMS 1216356.}}

\begin{document}

\maketitle

\begin{abstract}
In this paper we propose and analyze a class of simple Nystr\"om discretizations
of the hypersingular integral equation for the Helmholtz problem on domains of
the plane with smooth parametrizable boundary. The method depends on a parameter
(related to the staggering of two underlying grids) and we show that two choices
of this parameter produce convergent methods of order two, while all other
stable methods provide methods of order one. Convergence is shown for the
density (in uniform norm) and for the potential postprocessing of the solution.
Some numerical experiments are given to illustrate the performance of the
method.
\end{abstract}

\section{Introduction}

In this paper we propose and analyze a discretization of the hypersingular
integral equation for the Helmholtz equation on a smooth parametrizable simple
curve $\Gamma\subset \mathbb R^2$:
\begin{equation}\label{eq:0.1}
-\partial_{\nu}\int_\Gamma \partial_{\nu(\mathbf y)} H^{(1)}_0(k|\cdot-\mathbf
y|) \phi(\mathbf y)\mathrm d\Gamma(\mathbf y)= g \quad \mbox{on $\Gamma$.}
\end{equation}
Here $H^{(1)}_0$ is the Hankel function of the first kind and order zero, $k$ is
the wave number, $\partial_\nu$ is the normal derivative, and $g$ is data on
$\Gamma$.
The method is based on some simple ideas: 
\begin{itemize}
\item[(a)] The operator is first written as a bilinear integrodifferential form
acting on periodic functions
\begin{eqnarray*}
(\psi,\varphi) &\longmapsto &
\int_0^1 \int_0^1 H^{(1)}_0(k|\mathbf x(s)-\mathbf
x(t)|)\psi'(s)\varphi'(t)\mathrm d s\mathrm d t\\
& &-k^2\int_0^1 \int_0^1 H^{(1)}_0(k|\mathbf x(s)-\mathbf x(t)|)\mathbf
n(s)\cdot\mathbf n(t) \psi(s)\varphi(t)\mathrm ds\mathrm dt,
\end{eqnarray*}
$\mathbf x=(x_1,x_2)$ being a regular parametrization of $\Gamma$, and $\mathbf
n=(x_2',-x_1')$ the outward pointing normal vector.
\item[(b)] The principal part of the bilinear form (the one with the
derivatives) is formally approximated with a nonconforming Petrov-Galerkin
scheme,  using piecewise constant functions on two different uniform grids with
the same mesh-size $h$: $\{ (j-\frac12)h\}$ and $\{(j-\frac12+\varepsilon)h\}$.
Since the derivatives of piecewise constant functions are linear combinations of
Dirac delta distributions, that part of the bilinear form is just discretized
with the matrix
\[
\mathrm W_{i+1,j+1}-\mathrm W_{i,j+1}+\mathrm W_{i,j}-\mathrm W_{i+1,j},
\]
where $\mathrm W_{i,j}=  H^{(1)}_0(k|\mathbf b_i^\varepsilon-\mathbf b_j|)$, 
 $\mathbf b_i^\varepsilon:=\mathbf x((i-\smallfrac12+\varepsilon)h)$, and
$\mathbf b_j:=\mathbf x((j-\smallfrac12)h)$.
\item[(c)] The second part of the bilinear form (which has a weakly singular
logarithmic singularity) is discretized with the same Petrov-Galerkin scheme,
using midpoint quadrature to approximate the resulting integrals:
\begin{eqnarray*}
& & \hspace{-2cm}k^2H^{(1)}_0(k|\mathbf m_i^\varepsilon-\mathbf m_j|) \mathbf
n_i^\varepsilon\cdot\mathbf n_j \\
& & \approx k^2
\int_{(i-\frac12+\varepsilon)h}^{(i+\frac12+\varepsilon)h}
\int_{(j-\frac12)h}^{(j+\frac12)h} H^{(1)}_0(k|\mathbf x(s)-\mathbf
x(t)|)\mathbf n(s)\cdot\mathbf n(t) \mathrm ds\mathrm dt,
\end{eqnarray*}
where $\mathbf m_i^\varepsilon:=\mathbf x((i+\varepsilon)h)$, $\mathbf
n_i^\varepsilon:=h\,\mathbf n((i+\varepsilon)h)$ and $\mathbf m_j$ and $\mathbf
n_j$ are similarly defined.
\item[(d)] The right-hand side is tested with piecewise constant functions and
then midpoint quadrature is applied to all the resulting integrals.
\end{itemize}
As can be seen from the above formulas, this method leads to a very simple
discretization of \eqref{eq:0.1}, requiring no assembly process, no additional
numerical integration and no complicated data structures to handle the geometric
data.

The use of a two-grid Nystr\"om method for periodic logarithmic integral
equations goes back to the work of Jukka Saranen, Ian Sloan and their
collaborators \cite{saSc:1993, SaSl:1992,   SlBu:1992}. It was then discovered
that the values $\varepsilon=\pm 1/6$ provide superconvergent methods (of order
two) and that the values $\varepsilon=\pm 1/2$ lead to unstable discretizations.
The idea was further exploited in \cite{CeDoSa:2002}, showing that the methods
can be used on the weakly singular equations that appear in the Helmholtz
equation. The present paper shows how to transfer the same kind of ideas (and,
up to a point, the same type of analysis) to the hypersingular integral equation
for the Helmholtz equation. {\em The case of the Laplace hypersingular equation
is included in the present analysis}.

We will show that this discretization of the hypersingular integral operator is
stable in the $L^2$ norm for the underlying space of piecewise constant
functions as long as $\varepsilon\neq \pm 1/2$ (the value $\varepsilon=0$ is
excluded as a possibility from the very beginning, since it leads to evaluation
of the kernel functions on the singularity). We will also show that
$\varepsilon=\pm 1/6$ define methods of order two and that this order is
actually attained in a strong $L^\infty$ norm. The error analysis will be based
on Fourier techniques \cite{Arnold:1983, saSc:1993, SaVa:2002} combined with an
already quite extensive library of asymptotic expansions developed by two of the
authors of this paper with some other collaborators \cite{CeDoSa:2002,
DoRaSa:2008, DoSa:2001, DoSa:2001a}.

In a forthcoming paper \cite{DoLuSa:2012} we will show how to combine this
discretization method for the hypersingular equation with the original method
\cite{ CeDoSa:2002, saSc:1993, SaSl:1992, SlBu:1992} for the single layer
operator and with straightforward Nystr\"om discretization of the double layer
operator and its adjoint. This results in a compatible and
straightforward-to-code fully discrete Calder\'on Calculus for the two
dimensional Helmholtz equation on a finite number of disjoint smooth closed
curves. This discretization set has a strong flavor to low order Finite
Differences. This might make it an attractive option to build simple code for
scattering problems, when the simultaneous use of several boundary integral
operators is required. 

The paper is structured as follows. In Section \ref{sec:1} we present the method
for a class of periodic hypersingular equations that include \eqref{eq:0.1}
after parametrization. The method is then reinterpreted as a non-conforming
Petrov-Galerkin discretization with numerical quadrature.  In Section
\ref{sec:2}, we introduce the functional frame for the analysis of the method,
based on the theory of periodic pseudodifferential operators on periodic Sobolev
spaces. In Section \ref{sec:3} we present the stability result of this paper in
the form of an infimum-supremum condition.
In Sections \ref{sec:4} and \ref{sec:5}, we respectively give the consistency
and convergence error estimates  for the method. Section \ref{sec:7} contains
some numerical experiments, while we have gathered in Appendix \ref{sec:A} the
more technical proof of Proposition \ref{prop:A.4}.

\section{The equation and the method}\label{sec:1}

We are going to present the method for a hypersingular integral equation
associated to the two dimensional equation on a simple smooth curve. The
extension to a finite set of non-intersecting smooth curves is straightforward.
Since everything can be expressed with periodic integral equations, and the
analysis will be carried out at that level of generality, we start this work by
presenting the equation in that language.

\subsection{A class of integrodifferential operators and its discretization}

We consider two logarithmic kernel functions:
\begin{equation}\label{eq:1.0a}
V_\ell(s,t):=\smallfrac{\imath}{2\pi}A_\ell(s,t)\log(\sin^2
\pi(s-t))+K_\ell(s,t) ,\qquad \ell\in \{1,2\},
\end{equation}
where $A_1,A_2, K_1, K_2\in \mathcal C^\infty(\mathbb R^2)$ are 1-periodic in
both variables and
\begin{equation}\label{eq:1.0b}
A_1(s,s)\equiv \alpha\neq 0.
\end{equation}
We consider the associated integral operators
\begin{equation}\label{eq:1.0c}
\mathrm V_j \varphi:=\int_0^1 V_j(\punto,t)\varphi(t)\mathrm dt, 
\end{equation}
and the operator
\begin{equation}\label{eq:1.1}
\mathrm W:=-\mathrm D \mathrm V_1\mathrm D+\mathrm V_2, \qquad \mathrm
D\varphi:=\varphi'.
\end{equation}
Let us consider the space of smooth 1-periodic functions $\mathcal D:=\{
\varphi\in \mathcal C^\infty(\mathbb R)\,:\, \varphi(1+\punto)=\varphi\}.$
An important hypothesis is injectivity:
\begin{equation}\label{eq:1.2}
\varphi \in \mathcal D, \quad \mathrm W \varphi = 0 \qquad \Longrightarrow
\qquad \varphi =0.
\end{equation}
As we will see later on, this injectivity condition is enough to prove
invertibility of $\mathrm W$ in a wide range of Sobolev spaces.
Given a smooth 1-periodic function $g$, we look for $\varphi$ such that
\begin{equation}\label{eq:1.3}
\mathrm W\varphi=g.
\end{equation}

The discretization method uses four sets of discretization points. Let $N$ be a
positive integer, $h:=1/N$, and
\begin{equation}\label{eq:1.4}
s_i=(i-\smallfrac12)h, \quad t_i:=i\,h, \quad s_{i+\varepsilon}:=(i+\varepsilon
-\smallfrac12)h, \quad t_{i+\varepsilon}:=(i+\varepsilon)h, \quad i\in \mathbb
Z.
\end{equation}
(We will comment on $\varepsilon$ shortly.)
The discretization method looks for
\begin{equation}\label{eq:1.5}
(\varphi_0,\ldots,\varphi_{N-1})\in \mathbb C^N \quad\mbox{such that} \quad
\sum_{j=0}^{N-1} W_{ij}^\varepsilon \varphi_j = h g(t_{i+\varepsilon})\quad
i=0,\ldots,N-1,
\end{equation}
where
\begin{eqnarray}\nonumber
W_{ij}^\varepsilon&:=&
V_1(s_{i+1+\varepsilon},s_{j+1})-V_1(s_{i+\varepsilon},s_{j+1})-V_1(s_{
i+1+\varepsilon},s_j)+V_1(s_{i+\varepsilon},s_j)\\
& & + h^2 V_2(t_{i+\varepsilon},t_j).\label{eq:1.6}
\end{eqnarray}
Substitution of $\varepsilon$ by $\varepsilon+1$ produces the same method. The
option $\varepsilon \in \mathbb Z$ is not practicable, since it leads to
evaluations of the logarithmic kernels in their diagonal singularity. The method
for $\varepsilon=1/2$ ( $\varepsilon \in 1/2+\mathbb Z$) will not fit in our
analysis, that relies on stability properties of an $\varepsilon-$dependent
discretization of logarithmic operators that is unstable for $\varepsilon=1/2$.
(We will show numerical evidence that the value $\varepsilon=1/2$ is valid
though.)
All other methods will provide convergent schemes, with two superconvergent
cases. Namely, we will see that for smooth enough solutions, we can prove:
\[
\max_j | \varphi_j-\varphi(t_j)|=\mathcal O(h), \qquad \mbox{if } \varepsilon
\not \in \smallfrac12\mathbb Z
\]
(this excludes the non-practicable and unstable cases), and that
\[
\max_j | \varphi_j-\varphi(t_j)|=\mathcal O(h^2), \qquad \mbox{if }
\varepsilon\in \pm\smallfrac16+\mathbb Z.
\]
These results will be proved as Theorems \ref{the:5.3} and \ref{the:5.4}
respectively.

\subsection{A non-conforming Petrov-Galerkin method}

We next give some intuition on how to come up with the method
\eqref{eq:1.5}-\eqref{eq:1.6}. We can formally rewrite \eqref{eq:1.3} in
variational form
\begin{equation}\label{eq:1.7}
\int_0^1 \psi'(s) (\mathrm V_1\varphi')(s)\mathrm d s+\int_0^1 \psi(s) (\mathrm
V_2\varphi)(s)\mathrm d s=\int_0^1 \psi(s)g(s)\mathrm ds.
\end{equation}
Consider now the function $\chi_i$ that arises from 1-periodization of the
characteristic function of the interval $(s_i,s_{i+1})=(t_i-h/2,t_i+h/2)$, that
is,
\begin{equation}\label{eq:1.8}
\chi_i (1+\punto)=\chi_i, \qquad \chi_i\big|_{(s_i,s_{i+1})}\equiv 1, \qquad
\chi_i\big|_{[s_{i+1},s_{i+N}]}\equiv 0.
\end{equation}
We similarly define the functions $\chi_{i+\varepsilon}$ by periodizing the
characteristic functions of the intervals
$(s_{i+\varepsilon},s_{i+1+\varepsilon})=(t_{i+\varepsilon}-h/2,t_{i+\varepsilon
}+h/2)$. The weak derivatives of these functions can be expressed through the
use of Dirac delta distributions. In addition to understanding Dirac deltas as
periodic distributions (see Section \ref{sec:2.3}), we will admit the action of
Dirac deltas on any function that is continuous around the point where the delta
is concentrated. At the present stage, we only need to consider the functionals
\begin{equation}\label{eq:1.9}
\{ \delta_{s_i},\varphi\}:=\varphi(s_i), \qquad
\{\delta_{s_{i+\varepsilon}},\varphi\}:=\varphi(s_{i+\varepsilon}),
\end{equation}
acting on any 1-periodic function $\varphi$ that is smooth in a neighborhood of
$s_i$ and $s_{i+\varepsilon}$. Admitting formally that
$\chi_i'=\delta_{s_i}-\delta_{s_{i+1}}$ and
$\chi_{i+\varepsilon}'=\delta_{s_{i+\varepsilon}}-\delta_{s_{i+1+\varepsilon}}$,
we consider a non-conforming Petrov-Galerkin discretization of \eqref{eq:1.7}:
\begin{equation}\label{eq:1.10}
\varphi_h:=\sum_{j=0}^{N-1} \varphi_j \chi_j \quad\mbox{such that} \quad
\sum_{j=0}^N W_{ij}^{\varepsilon,\circ} \varphi_j=
\int_{t_{i+\varepsilon}-\frac{h}2}^{t_{i+\varepsilon}+\frac{h}2} g(s)\mathrm ds,
\qquad i=0,\ldots,N-1,
\end{equation}
where
\begin{eqnarray}\nonumber
W_{ij}^{\varepsilon,\circ}&:=& 
V_1(s_{i+1+\varepsilon},s_{j+1})-V_1(s_{i+\varepsilon},s_{j+1})-V_1(s_{
i+1+\varepsilon},s_j)+V_1(s_{i+\varepsilon},s_j)\\
&&+\int_{t_{i+\varepsilon}-\frac{h}2}^{t_{i+\varepsilon}+\frac{h}2}\int_{
t_j-\frac{h}2}^{t_j+\frac{h}2} V_2(s,t)\mathrm ds\mathrm dt.
\label{eq:1.11}
\end{eqnarray}
Note that
\[
\sum_{j=0}^{N-1} \Big(
\int_{t_{i+\varepsilon}-\frac{h}2}^{t_{i+\varepsilon}+\frac{h}2}\int_{t_j-\frac{
h}2}^{t_j+\frac{h}2} V_2(s,t)\mathrm ds\mathrm dt
\Big) \varphi_j=
\int_{t_{i+\varepsilon}-\frac{h}2}^{t_{i+\varepsilon}+\frac{h}2} (\mathrm V_2
\varphi_h)(s)\mathrm ds
\]
and that the leading term in $W_{ij}^{\varepsilon,\circ}$ can be understood as
the action
\begin{equation}\label{eq:1.20}
\{ \delta_{s_{i+\varepsilon}} -\delta_{s_{i+1+\varepsilon}}, \mathrm
V_1(\delta_{s_j}-\delta_{s_{j+1}})\}=\{ \chi_{i+\varepsilon}',\mathrm
V_1\chi_j'\}.
\end{equation}
The method \eqref{eq:1.5}-\eqref{eq:1.6} is recovered if we use midpoint
integration for all integrals in \eqref{eq:1.10}-\eqref{eq:1.11}.

\subsection{Relation to the Helmholtz equation}\label{sec:1.3}

Let $\mathbf x:\mathbb R \to \mathbb R^2$ be a smooth 1-periodic function such
that
$|\mathbf x'(s)|\neq 0$ for all $s$, and $\mathbf x(s)\neq\mathbf x(t)$ if
$|s-t|<1$.
The range of $\mathbf x$ is then a smooth closed curve $\Gamma$ in the plane. 
Let $\mathbf n(t)$ be the outward pointing normal vector at $\mathbf x(t)$ with
$|\mathbf n(t)|=|\mathbf x'(t)|$. Given a periodic function $\varphi$, we define
\begin{eqnarray}\label{eq:1.13}
\mathbb R^2\setminus\Gamma \ni \mathbf z\mapsto U(\mathbf z) &:=&
\frac\imath4\int_0^1 \nabla_{\mathbf y} H^{(1)}_0(k|\mathbf z-\mathbf
y|)\Big|_{\mathbf y=\mathbf x(t)}\hspace{-5pt}\cdot\mathbf n(t)\,\varphi(t)
\mathrm dt\\
&=& \frac{\imath k}4\int_0^1 H^{(1)}_1(k|\mathbf z-\mathbf
x(t)|)\,\frac{(\mathbf z-\mathbf x(t))\cdot\mathbf n(t)}{|\mathbf z-\mathbf
x(t)|} \varphi(t)\mathrm dt,\nonumber
\end{eqnarray}
where $H^{(1)}_0$ and $H^{(1)}_1$ are the Hankel functions of the first kind and
orders $0$ and $1$ respectively. The function $U$ is an outgoing solution of the
Helmholtz equation
\begin{equation}\label{eq:1.14}
\Delta U+k^2 U=0 \quad\mbox{in $\mathbb R^2\setminus\Gamma$}, \qquad\qquad
\frac{\partial U}{\partial r}-\imath k U = o(r^{-1/2}) \quad \mbox{as
$r=|\mathbf z|\to\infty$},
\end{equation}
with the asymptotic limit (the Sommerfeld radiation condition) holding uniformly
in all directions:
 $U$ is the double layer potential with (parametrized) density $\varphi$  (see
\cite[Section 2.1]{HsWe:2008}). The double layer potential is discontinuous
across $\Gamma$ but its normal derivative on $\Gamma$ coincides from both sides.
If we define
\[
(\mathrm W\varphi)(s) :=-(\nabla U(\mathbf x(s) )\cdot\mathbf n(s),
\]
then,
\begin{eqnarray}\nonumber
(\mathrm W\varphi)(s)&=&-\frac\imath4 \frac{\mathrm d}{\mathrm ds} \int_0^1
H^{(1)}_0(k |\mathbf x(s)-\mathbf x(t)|)\varphi'(t)\mathrm dt \\
& & -\frac{\imath k^2}4\int_0^1  H^{(1)}_0(k |\mathbf x(s)-\mathbf x(t)|)
\mathbf n(s)\cdot\mathbf n(t) \,\varphi(t)\mathrm d t.\label{eq:1.15}
\end{eqnarray}
This is just the parametrized form of a well known formula: see \cite[Theorem
3.3.22]{SauSch:2011} or \cite[Exercise 9.6]{McLean:2000}. Using the asymptotic
behavior of Hankel functions close to the singularity, it can be shown that the
weakly singular kernel $H^{(1)}_0(k |\mathbf x(s)-\mathbf x(t)|)$ can be
decomposed as in \eqref{eq:1.0a} with $A_1(s,s)\equiv \imath/2$ (see
\eqref{eq:1.0b}). Therefore, the
 operator $\mathrm W$ in \eqref{eq:1.15} fits in the frame of \eqref{eq:1.1}.

{\em The operator $\mathrm W$ satisfies the injectivity condition \eqref{eq:1.2}
if and only if $k^2$ is not a Neumann eigenvalue of the Laplace operator in the
interior of $\Gamma$} \cite[Section 2.1]{HsWe:2008}. In those cases, the
solution of the interior-exterior Helmholtz equation \eqref{eq:1.14} with
Neumann boundary condition $\partial_\nu u= f$ can be represented with the
double layer ansatz \eqref{eq:1.13},  $\varphi$ being a solution of equation
\eqref{eq:1.3} with $g:=-f\circ\mathbf x$. If we then apply the discretization
\eqref{eq:1.5}-\eqref{eq:1.6}, we can construct a fully discrete potential
representation, when we substitute $\varphi$ by
$\varphi_h:=\sum_{j=0}^{N-1}\varphi_j\chi_j$ in \eqref{eq:1.13} and apply
midpoint integration:
\begin{equation}\label{eq:1.17}
U_h(\mathbf z):=\frac{\imath k}4 \,h\, \sum_{j=0}^{N-1} H^{(1)}_1(k|\mathbf
z-\mathbf x(t_j)|)\, \frac{(\mathbf z-\mathbf x(t_j))\cdot\mathbf
n(t_j)}{|\mathbf z-\mathbf x(t_j)|}\, \varphi_j.
\end{equation}

\section{Functional frame}\label{sec:2}

\subsection{Asymptotics of hypersingular operators}

Consider the space $\mathcal D$ of periodic $\mathcal C^\infty$ complex valued
functions of one variable, endowed with the metric that imposes uniform
convergence of all derivatives \cite[Section 5.2]{SaVa:2002}. A periodic
distribution is an element of $\mathcal D'$, the dual space of $\mathcal D$.
Given $u \in \mathcal D'$, we consider its Fourier coefficients
\begin{equation}\label{eq:2.1}
\widehat u(m):= \langle u,\phi_m\rangle_{\mathcal D'\times\mathcal D}, \qquad
\phi_m(t):=\exp(2\pi\imath m t), \qquad m \in \mathbb Z.
\end{equation}
The periodic Sobolev space of order $r\in \mathbb R$ is
\begin{equation}\label{eq:2.2}
H^r:=\{ u\in \mathcal D'\,:\, \| u\|_r<\infty\}, \qquad \mbox{where}\qquad \|
u\|_r^2:=|\widehat u(0)|^2+\sum_{m\neq 0} |m|^{2r}|\widehat u(m)|^2.
\end{equation}
(From here on, the symbol $\sum_{m\neq 0}$ refers to a sum over all integers
except zero.) An extensive treatment of these spaces can be found in the
monograph \cite{SaVa:2002}. Let us just mention that $H^p\subset H^q$ for $p>q$,
with dense and compact injection. Also, $H^0$ can be identified with the space
of 1-periodic functions that are locally square integrable or equivalently, with
the 1-periodization of $L^2(0,1)$.

We say that an operator $\mathrm L:\mathcal D'\to\mathcal D'$ is a periodic
pseudodifferential operator of order $n$, and we write for short $\mathrm L\in
\mathcal E(n)$, when $\mathrm L: H^r \to H^{r-n}$ for all $r$. It then follows
from \cite[Paragraph 7.6.1]{SaVa:2002} that the logarithmic operators
\eqref{eq:1.0c} can be extended to act on all periodic distributions and,
consequently, so can $\mathrm W$. Moreover, $\mathrm V_1,\mathrm V_2 \in
\mathcal E(-1)$ and $\mathrm W\in \mathcal E(1)$.

A first group of pseudodifferential operators that we will use extensively is
that of multiplication operators. Given $a\in \mathcal D$ we define the operator
$\boldsymbol a \in \mathcal E(0)$ by
$\boldsymbol a u:=a\, u$.
The periodic Hilbert transform
\begin{equation}\label{eq:2.3}
\mathrm H u:=\sum_{m\neq 0} \mathrm{sign}(m) \widehat u(m) \phi_m
\end{equation}
is clearly a periodic pseudodifferential operator of order zero. We also
consider the operators for $n \in \mathbb Z$:
\begin{equation}\label{eq:2.4}
\mathrm D_n u:=\sum_{m\neq 0} (2\pi \imath m)^n \widehat u(m)\phi_m, \qquad
\mathrm D_n \in \mathcal E(n).
\end{equation}
It is easy to note that  $\mathrm D_1 =\mathrm D$ is the differentiation
operator and $\mathrm D_{-1}$ is a weak form of the following
antidifferentiation operator
\[
(\mathrm D_{-1}u)(s)=\int_0^s (u(t)-\widehat u(0))\mathrm dt \qquad \forall u\in
H^0.
\]
In the next lemma we collect some elementary properties of these operators.

\begin{lemma}\label{lemma:2.1}
The following properties of the operators $\mathrm D_n$ in \eqref{eq:2.4} hold:
\begin{enumerate}
\item[\rm{(a)}] $\mathrm D_0 u=u-\widehat u(0)$ for all $u$.
\item[\rm{(b)}] For all $n$ and $u$, $\mathrm D_{-n}\mathrm D_n u=u-\widehat
u(0)$.
\item[\rm{(c)}] For all $n,m$, $\mathrm D_n \mathrm D_m = \mathrm D_{n+m}$.
\item[\rm{(d)}] For all $n$, $\mathrm D_n\mathrm H =\mathrm H \mathrm D_n$.
\item[\rm{(e)}] For all $a\in \mathcal D$, $\mathrm D_1\boldsymbol a =
\boldsymbol a \mathrm D_1+\boldsymbol a'$.
\end{enumerate}
\end{lemma}

The following results show how logarithmic operators and the hypersingular
operators $\mathrm W$ can be represented up to operators of arbitrarily negative
order as a linear combination of compositions of the simple operators given
above.

\begin{proposition}\label{prop:2.2} Let 
\begin{equation}\label{eq:2.30}
\mathrm V u :=\frac{\imath}{2\pi}\int_0^1 A(\punto,t)
\log(\sin^2(\pi(\punto-t)))\,u(t)\mathrm dt+\int_0^1 K(\punto,t)\,u(t)\mathrm
dt,
\end{equation}
where $A,K\in \mathcal C^\infty(\mathbb R^2)$ are 1-periodic in each variable.
Then there exists a sequence $\{ a_n\}_{n\ge 1}\subset \mathcal D$ such that
for all $M$,
\begin{equation}\label{eq:2.31}
\mathrm V = \sum_{n=1}^{M-1} \boldsymbol a_n \mathrm H \mathrm D_{-n} + \mathrm
K_M, \qquad \mathrm K_M\in \mathcal E(-M).
\end{equation}
Moreover $a_1(s)= A(s,s)$.
\end{proposition}

\begin{proof}
See \cite[Proposition A.1]{DoSa:2001} or similar arguments in \cite[Chapter
6]{SaVa:2002}. 
\end{proof}

\begin{proposition}\label{prop:2.3}
Let $\mathrm W$ be the operator in \eqref{eq:1.1}. Then, there exists a sequence
$\{b_n\}_{n\ge 0} \subset \mathcal D$ such that for all $M$,
\[
\mathrm W = - \alpha \mathrm H \mathrm D_1 + \sum_{n=0}^{M-1} \boldsymbol b_n
\mathrm H \mathrm D_{-n} + \mathrm K_M, \qquad \mathrm K_M\in \mathcal E(-M),
\]
where $0\neq \alpha\in \mathbb C$ is the constant in \eqref{eq:1.0b}.
\end{proposition}

\begin{proof}
This is just a direct consequence of Proposition \ref{prop:2.2} and Lemma
\ref{lemma:2.1}.
\end{proof}

\begin{proposition}
Hypothesis \eqref{eq:1.2} implies invertibility of $\mathrm W : H^r \to H^{r-1}$
for all $r$.
\end{proposition}

\begin{proof}
Consider the lowest order expansion of Proposition \ref{prop:2.3}, namely
$\mathrm W= -\alpha \mathrm H \mathrm D + \mathrm K_0$ with $\mathrm K_0 \in
\mathcal E(0)$. 
It is clear that $-\alpha\mathrm H \mathrm D: H^r\to H^{r-1}$ is Fredholm of
index zero, and therefore so is $\mathrm W$.

Let now $u\in H^r$ be such that $\mathrm W u=0$. Applying that $\mathrm H^2
v=v-\widehat v(0)$ for all $v$, and Lemma \ref{lemma:2.1} (b) and (d), it
follows that
\[
0 = \mathrm H \mathrm D_{-1} \mathrm W u = -\alpha u + w, \qquad
\mbox{where}\quad w:=\alpha\widehat u(0)+\mathrm H \mathrm D_{-1} \mathrm K_0 u
\in H^{r+1}. 
\] 
This means that $u \in H^s$ for all $s$ and therefore $u\in \mathcal D$.
Hypothesis \eqref{eq:1.2} implies then that $u=0$. Therefore $\mathrm W: H^r \to
H^{r-1}$ is injective and, by the Fredholm Alternative, it is invertible.
\end{proof}

\subsection{Variational formulation of the discrete method}\label{sec:2.3}

For a fixed $t\in \mathbb R$ we can define the Dirac delta distribution
$\delta_t$ by its action on elements of $\mathcal D$, $\langle
\delta_t,\phi\rangle_{\mathcal D'\times\mathcal D}=\phi(t)$. Using the Sobolev
embedding theorem \cite[Lemma 5.3.2]{SaVa:2002}, we can prove that $\delta_t \in
H^r$ for all $r<-1/2$. However, this does not allow us to apply the Dirac delta
to functions that are piecewise smooth on points where they do not have jumps.
If $u$ is a 1-periodic function that is continuous in a neighborhood of $t$, we
will write $\{ \delta_t,u\}=u(t)$. Note that in general this is not a duality
product $\mathcal D'\times \mathcal D$ or $H^r\times H^{-r}$. With this
definition, we can admit the Dirac deltas $\delta_{s_i}$ and
$\delta_{s_{i+\varepsilon}}$ in \eqref{eq:1.9}, as well as formula
\eqref{eq:1.20}.

Let $\mathbb P_0$ be the space of constant functions.
We then introduce three $N$-dimensional spaces:
\begin{eqnarray*}
S_h &:=& \mathrm{span}\{ \chi_i\,:\,i=0,\ldots,N-1\}=\{ u_h \in
H^0\,:\,u_h|_{(s_i,s_{i+1})}\in \mathbb P_0\quad \forall i\},\\
S_{h,\varepsilon} &:=& \mathrm{span}\{
\chi_{i+\varepsilon}\,:\,i=0,\ldots,N-1\}=\{ u_h \in
H^0\,:\,u_h|_{(s_{i+\varepsilon},s_{i+1+\varepsilon})}\in \mathbb P_0\quad
\forall i\},\\
S_h^{-1}&:=& \mathrm{span}\{ \delta_{s_i}\,:\,i=0,\ldots,N-1\} =\{ u_h'\,:\, u_h
\in S_h\}\oplus\mathrm{span}\,\{ d_h\},
\end{eqnarray*}
where
\begin{equation}\label{eq:2.5}
d_h:=h\sum_{j=0}^{N-1} \delta_{s_j}.
\end{equation}
Finally, we consider the discrete operators
\begin{equation}\label{eq:2.6}
Q_h^{-1} u:= h \sum_{j=0}^{N-1} u(t_j) \delta_{t_j} \qquad
Q_{h,\varepsilon}^{-1} u:=
h \sum_{j=0}^{N-1} u(t_{j+\varepsilon}) \delta_{t_{j+\varepsilon}},
\end{equation}
that are well defined for all periodic functions that are continuous around
$\cup\{ t_j\}=h\mathbb Z$ and $\cup\{ t_{j+\varepsilon}\}=h(\varepsilon+ \mathbb
Z)$ respectively. In particular, we can apply $Q_h^{-1}$ to elements of $S_h$
and $Q_{h,\varepsilon}^{-1}$ to elements of $S_{h,\varepsilon}$.

\begin{proposition}\label{prop:2.5}
Let $\varepsilon\not\in\mathbb Z$ and let $g$ be continuous in a neighborhood of
$h(\varepsilon+\mathbb Z)$. Then the discrete variational problem
\begin{equation}\label{eq:2.7}
\begin{array}{l} \mbox{find } \varphi_h \in S_h  \mbox{ such that}\\ 
\{  \psi_h',\mathrm V_1 \varphi_h'\}+
\{Q_{h,\varepsilon}^{-1} \psi_h,\mathrm V_2Q_h^{-1}\varphi_h\} = \{
Q_{h,\varepsilon}^{-1}\psi_h,g\}\quad \forall\psi_h\in S_{h,\varepsilon},
\end{array}
\end{equation}
is equivalent to looking for $\varphi_h=\sum_{j=0}^{N-1} \varphi_j \chi_j$,
where  equations \eqref{eq:1.5}-\eqref{eq:1.6} are satisfied.
\end{proposition}

\begin{proof}
Given that both $S_h$ and $S_{h,\varepsilon}$ are $N$-dimensional, the problem
\eqref{eq:2.7} can be reduced to a $N\times N$ linear system, after choosing a
basis for each of the spaces. The result follows then from several simple
observations. First of all  $\chi_j'=\delta_{s_j}-\delta_{s_{j+1}}$ and
$\chi_{i+\varepsilon}'=\delta_{s_{i+\varepsilon}}-\delta_{s_{i+1+\varepsilon}}$
and $\{ \delta_{s_{i+\varepsilon}},\mathrm
V_1\delta_{s_j}\}=V_1(s_{i+\varepsilon},s_j)$. Also
\[
 \{ Q_{h,\varepsilon}^{-1} \chi_{i+\varepsilon},\mathrm V_2 Q_h^{-1}\chi_j\}=h^2
 \{ \delta_{t_{i+\varepsilon}},\mathrm V_2
\delta_{t_j}\}=h^2V_2(t_{i+\varepsilon},t_j).
\]
Finally $\{Q_{h,\varepsilon}^{-1}\chi_{i+\varepsilon},g\}=h
g(t_{i+\varepsilon})$.
\end{proof}

The non-conforming Petrov-Galerkin discretization of \eqref{eq:1.7} given in
\eqref{eq:1.10}-\eqref{eq:1.11} is equivalent to the discrete variational
problem
\begin{equation}\label{eq:2.8}
\begin{array}{l} \mbox{find } \varphi_h \in S_h  \mbox{ such that}\\
\{ \psi_h',\mathrm V_1 \varphi_h'\}+
( \psi_h,\mathrm V_2 \varphi_h) = (\psi_h,g)\quad \forall\psi_h\in
S_{h,\varepsilon},
\end{array}
\end{equation}
where $(u,v)=\int_0^1 u(t) v(t)\mathrm dt.$ In the sequel $(\punto,\punto)$ will
be used to denote this bilinear form in $H^0$ (so that $(\overline
u,u)=\|u\|_0^2$) and its extension to a duality product $H^{-1}\times H^1$, so
that for any $r\in \mathbb R$.
\begin{equation}\label{eq:2.32}
\| v\|_{-r}=\sup_{0\neq u \in H^r}\frac{|(u,v)|}{\|u\|_r} \qquad v \in H^r.
\end{equation}
{\em We will always take adjoints with respect to this bilinear form, thus
avoiding conjugation.}

\section{Stability analysis via an inf-sup condition}\label{sec:3}

Consider now the bilinear form $w_h: S_{h,\varepsilon}\times S_h \to \mathbb C$ 
associated to the problem \eqref{eq:2.7}, namely
\begin{equation}\label{eq:3.1}
w_h(\psi_h,\varphi_h):=\{  \psi_h',\mathrm V_1 \varphi_h'\}+
\{Q_{h,\varepsilon}^{-1} \psi_h,\mathrm V_2Q_h^{-1}\varphi_h\}.
\end{equation}
The aim of this section is the proof of the following result, that in particular
implies that problem \eqref{eq:2.7} (and by Proposition \ref{prop:2.5} also the
method \eqref{eq:1.5}-\eqref{eq:1.6}) has a unique solution for small enough
$h$.

\begin{theorem}[Stability]\label{the:3.1}
There exist $h_0$ and $\beta_\varepsilon>0$ such that for $h \le h_0$,
\[
\sup_{0\neq \psi_h\in S_{h,\varepsilon}}\frac{| w_h(\psi_h,\varphi_h)|}{\|
\psi_h\|_0} \ge \beta_\varepsilon\|\varphi_h\|_0 \qquad \forall \varphi_h \in
S_h.
\]
\end{theorem}

\subsection{Stability of the non-conforming PG method}

We start by considering the bilinear form associated to problem \eqref{eq:2.8}
\begin{equation}\label{eq:3.20}
w_h^\circ(\psi_h,\varphi_h):=\{  \psi_h',\mathrm V_1 \varphi_h'\}+
( \psi_h,\mathrm V_2\varphi_h),
\end{equation}
the operator $\mathrm A \varphi:=\mathrm W \mathrm
D_{-1}\varphi+\widehat\varphi(0)\mathrm W 1$ and its adjoint $\mathrm A^*$. Note
that if $\psi_h\in S_{h,\varepsilon}$, then $\mathrm A^* \psi_h$ is a smooth
function except at the discontinuity points of $\psi_h$.

\begin{lemma}\label{lemma:3.9}
\[
w_h^\circ(\psi_h,\varphi_h)=\{\varphi_h',\mathrm
A^*\psi_h\}+\widehat\varphi_h(0)(1,\mathrm A^*\psi_h) \qquad \forall \psi_h\in
S_{h,\varepsilon}, \, \varphi_h \in S_h.
\]
\end{lemma}

\begin{proof}
A direct computation, using the fact that $\mathrm D_{-1}^*=-\mathrm D_{-1}$ and
Lemma \ref{lemma:2.1}(b) shows that
\begin{equation}\label{eq:3.21}
\mathrm A^* \psi=-\mathrm D_{-1}\mathrm W^*\psi+\widehat{\mathrm
W^*\psi}(0)=\mathrm V_1^*\psi'-\widehat{\mathrm V_1^*\psi'}(0)-\mathrm
D_{-1}\mathrm V_2^*\psi+\widehat{\mathrm V_2^*\psi}(0).
\end{equation}
Noticing that $\{\varphi_h',1\}=(\varphi_h',1)=0$, it follows from
\eqref{eq:3.21} and Lemma \ref{lemma:2.1}(b) that
\begin{eqnarray*}
\{\varphi_h',\mathrm A^*\psi_h\} \!\!&=& \!\!\{\varphi_h',\mathrm
V_1^*\psi_h'\}+(\mathrm D_{-1}\varphi_h',\mathrm V_2^*\psi_h)\\
\!\!&=& \!\! \{\varphi_h',\mathrm V_1^*\psi_h'\}+(\varphi_h,\mathrm
V_2^*\psi_h)-\widehat\varphi_h(0)(1,\mathrm V_2^*\psi_h)=
w_h^\circ(\psi_h,\varphi_h)-\widehat\varphi_h(0)\widehat{\mathrm
V_2^*\psi_h}(0).
\end{eqnarray*}
At the same time, integrating in \eqref{eq:3.21}, it follows that $(1,\mathrm
A^*\psi)=\widehat{\mathrm V_2^*\psi}(0)$, which finishes the proof.
\end{proof}

\begin{lemma}\label{lemma:3.7}
$\|d_h-1\|_{-1}\le \pi h$.
\end{lemma}

\begin{proof}
Since
for all $u\in H^1$,
\begin{eqnarray*}
|(1-d_h,u)_0| &=& \Big|\int_0^1 u(t)\mathrm dt-h\sum_{j=0}^{N-1} u(s_j)\Big|\le
\sum_{j=0}^{N-1} \Big| \int_{s_j-\frac{h}2}^{s_j+\frac{h}2} u(t)\mathrm dt-
hu(s_j)\Big|\\
&\le & \frac{h}2\sum_{j=0}^{N-1} \int_{s_j-\frac{h}2}^{s_j+\frac{h}2}
|u'(t)|\mathrm dt\le \frac{h}2 \| u'\|_0\le \pi h \| u\|_1,
\end{eqnarray*} 
the result follows from \eqref{eq:2.32}.
\end{proof}

\begin{lemma}\label{lemma:3.10}
For $\varepsilon \not\in \smallfrac12\mathbb Z$, there exist positive constants
$\beta_\varepsilon$ and $C_\varepsilon$ such that
\begin{equation}\label{eq:3.22}
\sup_{0\neq \psi_h \in S_{h,\varepsilon}} \frac{|\{ \delta_h,\mathrm
A^*\psi_h\}|}{\|\psi_h\|_0}\ge \beta_\varepsilon \|\delta_h\|_{-1}\qquad \forall
\delta_h \in S_h^{-1},
\end{equation}
and
\begin{equation}\label{eq:3.23}
|(1,\mathrm A^*\psi_h)-\{d_h,\mathrm A^*\psi_h\}|\le C_\varepsilon
h\|\psi_h\|_0\qquad \forall \psi_h\in S_{h,\varepsilon}.
\end{equation}
\end{lemma}

\begin{proof} Using Proposition \ref{prop:2.3} (asymptotic expansion of $\mathrm
W$), it is simple to see that $\mathrm A=-\alpha\mathrm H+\mathrm K$, where
$\mathrm K\in \mathcal E(-1)$. This places us in the hypotheses of
\cite[Proposition 8]{CeDoSa:2002}, which proves \eqref{eq:3.22}. For the second
estimate, we write $\mathrm A^*=- \alpha\mathrm H + \mathrm K^*$ and decompose
\[
(1,\mathrm A^*\psi_h)-\{d_h,\mathrm A^*\psi_h\}=- \alpha\Big((1,\mathrm
H\psi_h)-h\sum_{j=0}^{N-1} (\mathrm H\psi_h)(s_j)\Big)+(1-d_h,\mathrm K\psi_h).
\]
The first term can be bounded using \cite[Lemma 3]{CeDoSa:2002}
\[
\Big| (1,\mathrm H\psi_h)-h\sum_{j=0}^{N-1} (\mathrm H\psi_h)(s_j)\Big|\le
C_\varepsilon h \|\psi_h\|_0 \qquad \forall  \psi_h\in S_{h,\varepsilon}.
\]
The second one follows from Lemma \ref{lemma:3.7},
\[
|(1-d_h,\mathrm K\psi_h)|\le \| 1-d_h\|_{-1}\|\mathrm K^*\psi_h\|_1\le \pi h
\|\mathrm K^*\|_{H^0 \to H^1} \|\psi_h\|_0.
\]
\end{proof}

\begin{lemma}\label{lemma:3.8} There exist two positive constants $c_1, c_2$
such that
\[
c_1\| \varphi_h'+\widehat \varphi_h(0)\,d_h\|_{-1}\le \|\varphi_h\|_0\le c_2 \|
\varphi_h'+\widehat \varphi_h(0)\,d_h\|_{-1}\qquad \forall \varphi_h \in S_h.
\]
\end{lemma}

\begin{proof} The first bound is a simple consequence of the following
inequalities
\[
\| \varphi_h'\|_{-1}\le 2\pi \|\varphi_h\|_0, \qquad |\widehat\varphi_h(0)|\le
\|\varphi_h\|_0, \qquad \|d_h\|_{-1}\le 1+\pi h
\]
(the last inequality follows from Lemma \ref{lemma:3.7}.)
Note now that the operator $S_h\to S_h^{-1}$ given by $\varphi_h \mapsto
\varphi_h'+\widehat\varphi_h(0)d_h$ is injective. Its inverse is $S_h^{-1}\ni
\delta_h \mapsto \mathrm D_{-1} \delta_h +\widehat\delta_h(0)(1-\mathrm
D_{-1}d_h)$. Then the second inequality of the statement follows from the fact
that
\[
\| \mathrm D_{-1} \delta_h\|_{0}\le \smallfrac1{2\pi}\|\delta_h\|_{-1}, \quad
|\widehat\delta_h(0)|\le \|\delta_h\|_{-1}, \quad \|1-\mathrm D_{-1}d_h\|_0\le
1+\smallfrac1{2\pi}\|d_h\|_{-1}\le 1+\smallfrac{h}2,
\]
where we have applied Lemma \ref{lemma:3.7} again.
\end{proof}

\begin{proposition}\label{prop:3.11}
There exist $h_0$ and $\beta_\varepsilon>0$ such that for $h \le h_0$,
\[
\sup_{0\neq \psi_h\in S_{h,\varepsilon}}\frac{| w_h^\circ(\psi_h,\varphi_h)|}{\|
\psi_h\|_0} \ge \beta_\varepsilon\|\varphi_h\|_0 \qquad \forall \varphi_h \in
S_h.
\]
\end{proposition}

\begin{proof}
By Lemma \ref{lemma:3.9}, we can write
\[
w_h^\circ(\psi_h,\varphi_h)=\{\varphi_h'+\widehat\varphi_h(0)d_h,\mathrm
A^*\psi_h\}+\widehat\varphi_h(0)\Big((1,\mathrm A^*\psi_h)-\{d_h,\mathrm
A^*\psi_h\}\Big)
\]
Therefore, by Lemma \ref{lemma:3.10}, it follows that
\begin{eqnarray*}
\sup_{0\neq \psi_h\in S_{h,\varepsilon}}\frac{| w_h^\circ(\psi_h,\varphi_h)|}{\|
\psi_h\|_0} &\ge & 
\sup_{0\neq \psi_h\in S_{h,\varepsilon}}\frac{|
\{\varphi_h'+\widehat\varphi_h(0)d_h,\mathrm A^*\psi_h\}|}{\| \psi_h\|_0}-
C_\varepsilon\,h |\widehat\varphi_h(0)|\\
&\ge & \beta_\varepsilon\|\varphi_h'+\widehat\varphi_h(0)d_h\|_{-1} -
C_\varepsilon h \|\varphi_h\|_0\\
&\ge & \big( \beta_\varepsilon c_2^{-1}-C_\varepsilon h\big)\|\varphi_h\|_0
\qquad \forall \varphi_h \in S_h,
\end{eqnarray*}
where we have applied Lemma \ref{lemma:3.8} in the last inequality.
\end{proof}

\subsection{A perturbation argument}

Consider now the quadrature error
\begin{equation}\label{eq:3.2}
E_{ij}^\varepsilon:=\int_{Q_{ij}^\varepsilon} V_2(s,t)\mathrm ds\mathrm dt - h^2
V_2(t_{i+\varepsilon},t_j),  \qquad i,j\in \mathbb Z,
\end{equation}
where $Q_{ij}^\varepsilon:=(s_{i+\varepsilon},s_{i+1+\varepsilon})\times
(s_j,s_{j+1})$. Let
$E^\varepsilon$ be the $N\times N$ matrix whose entries are the values
$E_{ij}^\varepsilon$ for $i,j=0,\ldots,N-1$. In the sequel $|E|_p$ will denote
the $p$-norm of the matrix $E$ (for $p\in \{1,2,\infty\}$) and
$|E|_{\mathrm{Frob}}$ will denote its Frobenius norm.

\begin{proposition}\label{prop:3.2}
\[
|(\psi_h,\mathrm V_2\varphi_h)-\{Q_{h,\varepsilon}^{-1} \psi_h,\mathrm
V_2Q_h^{-1}\varphi_h\}|\le h^{-1} | E^\varepsilon|_2
\|\psi_h\|_0\|\varphi_h\|_0\qquad \forall \psi_h \in S_{h,\varepsilon}\,
\varphi_h \in S_h.
\]
\end{proposition}

\begin{proof}
If we decompose $\psi_h=\sum_{i=0}^{N-1} \psi_i \chi_{i+\varepsilon}$ and
$\varphi_h=\sum_{j=0}^{N-1} \varphi_j \chi_j$, it is easy to see that
\[
(\psi_h,\mathrm V_2\varphi_h)-\{Q_{h,\varepsilon}^{-1} \psi_h,\mathrm
V_2Q_h^{-1}\varphi_h\}=\sum_{i,j=0}^{N-1} \psi_i E_{ij}^\varepsilon \varphi_j.
\]
The result is then straightforward noticing that
$\|\psi_h\|_0=h^{1/2}|(\psi_0,\ldots,\psi_{N-1})|$ and $\|\varphi_h\|_0=h^{1/2}
|(\varphi_0,\ldots,\varphi_{N-1})|,$
where $|\punto|$ is the Euclidean norm in $\mathbb C^N$.
\end{proof}

In order to simplify some forthcoming arguments, let us restrict (without loss
of generality) $\varepsilon$ to be in $[-1/2,1/2]\setminus\{0\}$ (the
restriction $\varepsilon\neq\pm1/2$ is not needed for these arguments). 

\begin{lemma}\label{lemma:3.3}
There exists $C_\varepsilon$ such that for all $h$
\[
|E_{ij}^\varepsilon|\le C_\varepsilon h^2 |\log h| \qquad \forall i,j.
\]
Moreover $C_\varepsilon$ diverges like $\log |\varepsilon|^{-1}$ as
$|\varepsilon|\to 0$.
\end{lemma}

\begin{proof} 
Since $E^\varepsilon_{i,j\pm N}=E^\varepsilon_{ij}$, we can choose $|i-j|\le
N/2$ and then
\begin{equation}\label{eq:3.3o}
|s-t|\le h + h(\smallfrac{N}2+|\varepsilon|)\le \smallfrac32 h\,+\smallfrac12
\le \smallfrac34, \qquad (s,t)\in Q_{ij}^\varepsilon,
\end{equation}
as long as $h \le 1/6$. (This is not a restriction, since we are only missing
the values $1\le N\le 5$ that can be incorporated by modifying the constants in
the final bound.) Also
\begin{equation}\label{eq:3.3oo}
|\varepsilon|\, h \le |t_{i+\varepsilon}-t_j| \le \smallfrac34.
\end{equation}
Because of the form of the kernel function $V_2$ (see \eqref{eq:1.0a}), in the
diagonal strip
$D:=\{ (s,t)\,:\, |s-t|\le 3/4\}$,
we can bound
\begin{equation}\label{eq:3.3}
|V_2(s,t)|\le C_1\log |s-t|^{-1}+ C_2, \qquad (s,t)\in D.
\end{equation}
Therefore, by \eqref{eq:3.3oo},
\begin{equation}\label{eq:3.3a}
|V_2(t_{i+\varepsilon},t_j)| \le C_1 \log |\varepsilon|^{-1} + C_1\log h^{-1} +
C_2, \qquad \forall i,j.
\end{equation}
The choice of indices $|i-j|\le N/2$  ensures that $Q_{ij}^\varepsilon \subset
D$. If $\mathrm{dist}(Q_{ij}^\varepsilon,\{(s,s)\,:\,s\in \mathbb R\})\ge h/2$,
then by \eqref{eq:3.3},
\begin{equation}\label{eq:3.4}
\Big|\int_{Q_{ij}^\varepsilon} V_2(s,t)\mathrm ds\mathrm dt\Big|\le C_1 h^2
(\log 2+\log h^{-1})+C_2 h^2.
\end{equation}
If, on the other hand, $\mathrm{dist}({Q_{ij}^\varepsilon},\{(s,s)\,:\,s\in
\mathbb R\})\le h/2$, a simple geometric argument shows that
${Q_{ij}^\varepsilon}\subset \{ (s,t)\,:\, t\in (s_j,s_{j+1}), |s-t|< c\,h\}$,
where $c:=\sqrt2+1/2<e.$ Therefore
\begin{equation}\label{eq:3.5}
\Big|\int_{Q_{ij}^\varepsilon} V_2(s,t)\mathrm ds\mathrm dt\Big|\le 2C_1h
\int_0^{c\,h}\hspace{-10pt}\log u^{-1}\mathrm du +C_2h^2= 2C_1c h^2 (1-\log c
+\log h^{-1})+C_2h^2.
\end{equation}
We can now gather the bounds \eqref{eq:3.3a}, \eqref{eq:3.4} and \eqref{eq:3.5},
rearrange terms and take upper bounds to prove the result.
\end{proof}

\begin{lemma} \label{lemma:3.4}
There exists $C$ independent of $\varepsilon$ such that for all $h$
\[
|E_{ij}^\varepsilon|\le  C\, \frac{h^2}{|i-j|^2} \qquad \mbox{for $i,j$ such
that } 2\le |i-j|\le \frac{N}2.
\]
\end{lemma}

\begin{proof}
If $(s,t)\in Q_{ij}^\varepsilon$, $2\le |i-j|\le N/2$, and $|\varepsilon|\le
1/2$, then
\begin{equation}\label{eq:3.9}
|s-t|\ge |t_{i+\varepsilon}-t_j|-|t-t_j|-|s-t_{i+\varepsilon}| \ge h |i-j| -
(|\varepsilon|+1) h \ge \smallfrac14 h |i-j|.
\end{equation}
Recall first the definition of $D:=\{(s,t)\,:\,|s-t|\le 3/4\}$ given in the
proof of Lemma \ref{lemma:3.3}.
Taking derivatives of the kernel function $V_2$, we can write
\[
\frac{\partial^2V_2}{\partial s^2}(s,t)=\frac{B_1(s,t)}{|s-t|^2}+B_2(s,t),
\qquad \frac{\partial^2V_2}{\partial
t^2}(s,t)=\frac{B_3(s,t)}{|s-t|^2}+B_4(s,t),
\]
where  for $\ell\in\{1,2,3,4\}$ the functions $B_\ell\in\mathcal C^0(D)$ are
bounded. Using the error bound for the midpoint formula in two variables, we can
estimate
\begin{eqnarray*}
|E_{ij}^\varepsilon| & \le & \frac{h^4}{24} \max_{(s,t)\in
\overline{Q_{ij}^\varepsilon}} \Big( \Big| \frac{\partial^2V_2}{\partial
s^2}(s,t)\Big|+ \Big| \frac{\partial^2V_2}{\partial t^2}(s,t)\Big|\Big)\\
&\le & \frac{h^4}{24}\Big( \| B_2\|_{L^\infty}+\|B_4\|_{L^\infty} + 
(\|B_1\|_{L^\infty}+\|B_3\|_{L^\infty}) \max_{(s,t)\in
\overline{Q_{ij}^\varepsilon}} |s-t|^{-2}\Big)\\
&\le & C_1 h^4 + C_2\frac{h^2}{|i-j|^2}\le
\Big(\frac{C_1}4+C_2\Big)\frac{h^2}{|i-j|^2}, 
\end{eqnarray*}
where we have applied \eqref{eq:3.9} and the upper bound $|i-j| h\le 1/2$.
\end{proof}

\begin{lemma}\label{lemma:3.5} There exists $C_\varepsilon$ such that for all
$h$,
$|E^\varepsilon|_2\le C_\varepsilon h^{3/2}.$
\end{lemma}

\begin{proof} We first decompose the matrix
$E^\varepsilon=E^\varepsilon_{\mathrm{trid}}+E^\varepsilon_{\mathrm{off}}$,
where $E^\varepsilon_{\mathrm{trid}}$ gathers all tridiagonal terms (modulo $N$)
of $E^\varepsilon$
\[
E^\varepsilon_{\mathrm{trid},ij}:= \left\{ \begin{array}{ll} E_{ij}^\varepsilon,
& \mbox{$|i-j| \le 1$ or $|i-j|=N-1$},\\
0, & \mbox{otherwise.}
\end{array}\right.
\]
Using Lemma \ref{lemma:3.3} and the fact that $E^\varepsilon_{\mathrm{trid}}$
has only three non-vanishing elements in each row and column, it is easy to
estimate
$|E^\varepsilon_{\mathrm{trid}}|_1+ |E^\varepsilon_{\mathrm{trid}}|_\infty\le 3
C_\varepsilon h^2 |\log h|$.
Therefore, by the Riesz-Thorin theorem
\begin{equation}\label{eq:3.10}
|E^\varepsilon_{\mathrm{trid}}|_2\le
|E^\varepsilon_{\mathrm{trid}}|_1^{1/2}|E^\varepsilon_{\mathrm{trid}}|_\infty^{
1/2}\le 3 C_\varepsilon h^2 |\log h|.
\end{equation}
On the other hand, we can estimate the off-diagonal terms using Lemma
\ref{lemma:3.4} (recall that we can move indices so that $|i-j|\le N/2$)
\begin{eqnarray}\nonumber
|E^\varepsilon_{\mathrm{off}}|_2^2 &\le &
|E^\varepsilon_{\mathrm{off}}|_{\mathrm{Frob}}^2 \le\sum_{i=0}^{N-1} \sum_{2\le
|i-j|\le N/2} |E_{ij}^\varepsilon|^2\\ 
& \le& C^2 h^4 \sum_{i=0}^{N-1} \sum_{2\le |i-j|\le \infty} \frac1{|i-j|^4}=h^3
2 C^2 \sum_{k=2}^\infty \frac1{k^4}. \label{eq:3.11}
\end{eqnarray}
Gathering \eqref{eq:3.10} and \eqref{eq:3.11}, the result follows.
\end{proof}

\begin{proof}[Proof of Theorem \ref{the:3.1}.]
Note that
\[
w_h(\psi_h,\varphi_h)=w_h^\circ(\psi_h,\varphi_h)+\{Q_{h,\varepsilon}^{-1}
\psi_h,\mathrm V_2Q_h^{-1}\varphi_h\}-(\psi_h,\mathrm V_2\varphi_h).
\]
As a direct consequence of Proposition \ref{prop:3.2} and Lemma \ref{lemma:3.5},
we can bound
\begin{equation}\label{eq:3.30}
|(\psi_h,\mathrm V_2\varphi_h)-\{Q_{h,\varepsilon}^{-1} \psi_h,\mathrm
V_2Q_h^{-1}\varphi_h\}|\le C_\varepsilon
h^{1/2}\|\psi_h\|_0\|\varphi_h\|_0\qquad \forall \psi_h \in S_{h,\varepsilon},\,
\varphi_h \in S_h.
\end{equation}
The proof is thus a simple consequence of this bound and Proposition
\ref{prop:3.11}.
\end{proof}

\section{Consistency error analysis}\label{sec:4}

The analysis of the consistency error is based in the careful use of estimates
for quadrature error and the combination of asymptotic expansions of discrete
and continuous operators. We start this section with some technical results that
will be needed in the sequel.

\subsection{Estimates for quadrature error}

\begin{lemma}\label{lemma:4.1}
The following bounds hold for all $h$ and all $\varepsilon$:

\begin{enumerate}
\item[{\rm (a)}] $|(Q_{h,\varepsilon}^{-1}\psi_h,u)-(\psi_h,u)|\le \frac12(\pi
h)^2 \|\psi_h\|_0\|u \|_2$ for all $\psi_h \in S_{h,\varepsilon}$ and $u\in
H^2$.
\item[{\rm (b)}] $|(Q_{h,\varepsilon}^{-1}\psi_h,u)|\le \|\psi_h\|_0(\|u\|_0+\pi
h\|u\|_1)$ for all $\psi_h \in S_{h,\varepsilon}$ and $u\in H^1$.
\end{enumerate}
\end{lemma}

\begin{proof}
Using Taylor expansions, it is easy to prove the following well-known bound for
the midpoint formula
\[
\Big|\int_{c-\frac{h}2}^{c+\frac{h}2} u(t)\mathrm d t-hu(c)\Big|\le \frac{h^2}8
\int_{c-\frac{h}2}^{c+\frac{h}2} |u''(t)|\mathrm d t,
\]
from where
\begin{eqnarray*}
|(Q_{h,\varepsilon}^{-1}\psi_h,u)-(\psi_h,u)| &= & \Big|\sum_{j=0}^{N-1}
\psi_j\Big(\int_{t_{j+\varepsilon}-\frac{h}2}^{t_{j+\varepsilon}+\frac{h}2}
u(t)\mathrm dt-h u(t_{j+\varepsilon})\Big)\Big|\\
&\le & \frac{h^2}8 \int_0^1 |\psi_h(t)|\, |u''(t)|\mathrm dt\le
\frac{h^2}8\|\psi_h\|_0 (2\pi)^2 \|u\|_2. 
\end{eqnarray*}
This proves (a). To prove (b) we proceed similarly, showing first that
\begin{equation}\label{eq:4.40}
|(Q_{h,\varepsilon}^{-1}\psi_h,u)-(\psi_h,u)|\le \pi h \|\psi_h\|_0\|u
\|_1\qquad \forall\psi_h \in S_{h,\varepsilon},\,u\in H^1,
\end{equation}
and then applying the inverse triangle inequality.
\end{proof}

\begin{lemma}\label{lemma:4.2}
There exists $C_\varepsilon$ such that
\[
|\{ Q_{h,\varepsilon}^{-1}\psi_h,\mathrm V_2Q_h^{-1} u\}|\le C_\varepsilon
\|\psi_h\|_0 (\|u\|_0+h\|u\|_1) \qquad \forall \psi_h \in S_{h,\varepsilon}, \,
u \in H^1.
\]
\end{lemma}

\begin{proof}
Let $u_h:=\sum_{j=0}^{N-1} u(t_j)\chi_j\in S_h$ be the midpoint interpolate of
$u$ onto $S_h$. A direct estimate shows that
\begin{equation}\label{eq:4.1}
\|u_h\|_0 \le \|u\|_0+\|u-u_h\|_0 \le \|u\|_0+\smallfrac{\pi h}{\sqrt2}\|u\|_1.
\end{equation}
On the other hand, since $Q_h^{-1} u=Q_h^{-1}u_h$, it follows from
\eqref{eq:3.30} that
\begin{eqnarray*}
|\{ Q_{h,\varepsilon}^{-1}\psi_h,\mathrm V_2Q_h^{-1} u\}| &=&|\{
Q_{h,\varepsilon}^{-1}\psi_h,\mathrm V_2Q_h^{-1} u_h\}|\le |(\psi_h,\mathrm V_2
u_h)|+ C_\varepsilon h^{1/2} \|\psi_h\|_0\|u_h\|_0\\
&\le& ( \|\mathrm V_2\|_{H^0\to H^0}+C_\varepsilon h^{1/2} )
\|\psi_h\|_0\|u_h\|_0.
\end{eqnarray*}
Applying \eqref{eq:4.1}, the result follows.
\end{proof}

\subsection{Discrete operators and expansions}

The truncation operator for the Fourier series
\[
F_h u:=\sum_{m \in \Lambda_N} \widehat u(m)\phi_m \qquad
\mbox{where}\,\Lambda_N:=\{m\,:\,-N/2< m\le N/2\}
\]
gives optimal approximation properties in all Sobolev norms \cite[Theorem
8.2.1]{SaVa:2002}
\begin{equation}\label{eq:4.2}
\| F_h u-u\|_s\le (\sqrt2 h)^{r-s}\|u\|_r \qquad r\ge s.
\end{equation}
We can also define a discretization operator onto $S_h$ based on matching the
central Fourier coefficients
\[
D_h u\in S_h \qquad \mbox{such that} \qquad \widehat{D_h u}(m)=\widehat u(m)
\quad \forall  m \in \Lambda_N.
\]
This operator is based on a class of spline-trigonometric projectors introduced
in \cite{Arnold:1983}. Here we will use it as introduced in \cite{DoSa:2001a}.
The following property
\begin{equation}\label{eq:4.3}
Q_{h,1/2}^{-1} F_h \mathrm D =\mathrm D \, D_h
\end{equation}
is a consequence of \cite[Lemma 5]{CeDoSa:2002}. 

Consider the 1-periodic functions $\underline B_\ell$ such that $(-1)^\ell\,
\ell! \underline B_\ell$ restricted to $(0,1)$ is equal to the Bernoulli
polynomial of degree $\ell$ for all $\ell$. Consider also $\underline
C_\ell:=\mathrm H \underline B_\ell$. By comparing their Fourier coefficients
\cite[Section 3]{CeDoSa:2002}, it is easy to prove that
\begin{equation}\label{eq:4.30}
\underline C_1(t)=-\smallfrac1{2\pi\imath}\log(4\sin^2(\pi t)) \qquad \mbox{and
therefore} \quad \underline C_1(\pm \smallfrac16)=0.
\end{equation}
Note that $\pm 1/6+\mathbb Z$ are the only zeros of $\underline C_1$.

\begin{proposition}\label{prop:4.3}
Let $a_1,a_2\in \mathcal D$ and
\[
\mathrm V:=\boldsymbol a_1\mathrm H\mathrm D_{-1}+\boldsymbol a_2\mathrm
H\mathrm D_{-2}+\mathrm K_3 \qquad \mbox{where $\mathrm K_3\in \mathcal E(-3)$}.
\]
Let then $\mathrm L_1:=\boldsymbol a_1$ and $\mathrm L_2:=\boldsymbol a_1\mathrm
D-\boldsymbol a_2$, and consider the operators
\begin{eqnarray*}
R_h u &:=& \mathrm V u-\mathrm V Q_h^{-1} F_h u + h \underline
C_1(\punto/h)\mathrm L_1 F_h u,\\
T_h u &:=& \mathrm V u-\mathrm V Q_{h,1/2}^{-1} F_h u + h \underline
C_1(\punto/h+1/2)\mathrm L_1 F_h u+h^2 \underline C_2(\punto/h+1/2)\mathrm L_2
F_h u.
\end{eqnarray*}
Then
\begin{eqnarray}\label{eq:4.31}
\|R_h u\|_0+h\| R_h u\|_1&\le & C h^2 \|u\|_1 \qquad \forall u \in H^1,\\
\|T_h u\|_0+h\| T_h u\|_1&\le & C h^3 \|u\|_2 \qquad \forall u \in H^2.
\label{eq:4.32}
\end{eqnarray}
\end{proposition}

\begin{proof}
It is a direct consequence of \cite[Proposition 16]{CeDoSa:2002}.
\end{proof}

\begin{proposition}\label{prop:4.4}
Let $E_h u:= u -D_h u + h \underline B_1(\punto/h+1/2)F_h u'.$
Then
\[
\| E_hu\|_0+ h\|E_h u\|_1\le C h^2 \|u\|_2 \qquad \forall u \in H^2.
\]
\end{proposition}

\begin{proof} The bound for $\|E_h u\|_0$ is given in \cite[Proposition
1]{DoSa:2001}. The $H^1$ bound can be obtained with similar arguments (see the
proof of \cite[Proposition 16]{CeDoSa:2002}).
\end{proof}

\subsection{Consistency error}

In the definition of the bilinear form \eqref{eq:3.1}, we only admitted discrete
arguments. In this section we will admit a continuous second argument. The
definition is equally valid.

\begin{proposition}\label{prop:4.5}
Let $\varphi_h$ be the solution of \eqref{eq:2.7} with right-hand side
$g=\mathrm W \varphi$. Then
\[
| w_h(\psi_h,\varphi_h-\varphi)+ h  \underline C_1(\varepsilon)
(\psi_h,\boldsymbol a\,\varphi)| \le C h^2\|\psi_h\|_0 \|\varphi\|_3 \qquad
\forall \psi_h \in S_{h,\varepsilon},
\]
where $a(s):=A_2(s,s)$.
\end{proposition}

\begin{proof}
Note first that by definition of $\varphi_h$
\begin{eqnarray}\label{eq:4.4}
 w_h(\psi_h,\varphi_h-\varphi) &=& (Q_{h,\varepsilon}^{-1}\psi_h,\mathrm W
\varphi)-(\psi_h',\mathrm V_1\varphi')-\{ Q_{h,\varepsilon}^{-1}\psi_h ,\mathrm
V_2 Q_h^{-1} \varphi\}\\
\nonumber
&=& \underbrace{-(Q_{h,\varepsilon}^{-1}\psi_h,\mathrm D\mathrm V_1\mathrm D
\varphi)+(\psi_h,\mathrm D\mathrm V_1\mathrm D\varphi)}_{=:T_1}\\
\nonumber
& & + \underbrace{\{ Q_{h,\varepsilon}^{-1}\psi_h,\mathrm V_2Q_h^{-1}
(F_h\varphi-\varphi)\}}_{=:T_2}
+\underbrace{\{Q_{h,\varepsilon}^{-1}\psi_h,\mathrm V_2(\varphi-Q_h^{-1}
F_h\varphi)\}}_{=:T_3}.
\end{eqnarray}
In order to estimate $T_1$, we apply Lemma \ref{lemma:4.1}(a) with $u=\mathrm D
\mathrm V_1\mathrm D\varphi$ and note that $\mathrm D\mathrm V_1\mathrm D \in
\mathcal E(1)$, to obtain
\begin{equation}\label{eq:4.5}
|T_1|\le \smallfrac12 h^2\pi^2\|\psi_h\|_0 \|\mathrm D \mathrm V_1\mathrm
D\varphi\|_2\le C h^2 \|\psi_h\|_0 \|\varphi\|_3.
\end{equation}
To estimate $T_2$, we apply Lemma \ref{lemma:4.2} and the approximation
properties of $F_h$ \eqref{eq:4.2}, so that 
\begin{equation}\label{eq:4.6}
|T_2|\le \|\psi_h\|_0 (\| F_h\varphi-\varphi\|_0+ \pi h\|F_h
\varphi-\varphi\|_1) \le C h^2 \|\psi_h\|_0 \|\varphi\|_2.
\end{equation}
To bound $T_3$ we will apply Proposition \ref{prop:4.3} to the operator $\mathrm
V_2$ (see Proposition \ref{prop:2.2}).  
It is simple to verify that
\[
\{ Q_{h,\varepsilon}^{-1}\psi_h, \underline C_1(\punto/h) u\}=\underline
C_1(\varepsilon)\, (Q_{h,\varepsilon}^{-1}\psi_h,u) \qquad \forall \psi_h \in
S_{h,\varepsilon}.
\]
Then, by Proposition \ref{prop:4.3}, and denoting $a(s):=A_2(s,s)$,
\begin{eqnarray*}
T_3 &=&-h \underline C_1(\varepsilon) (Q_{h,\varepsilon}^{-1}\psi_h,\boldsymbol
a F_h \varphi)+(Q_{h,\varepsilon}^{-1}\psi_h,R_h \varphi)\\
&=& -h\underline C_1(\varepsilon) (\psi_h,\boldsymbol a\varphi)\\
& & + \underbrace{h\underline C_1(\varepsilon)
(\psi_h-Q_{h,\varepsilon}^{-1}\psi_h,\boldsymbol a\varphi)}_{=:T_{31}}
+\underbrace{ h\underline C_1(\varepsilon)
(Q_{h,\varepsilon}^{-1}\psi_h,\boldsymbol a
(\varphi-F_h\varphi))}_{=:T_{32}}+\underbrace{(Q_{h,\varepsilon}^{-1}\psi_h,R_h
\varphi)}_{=:T_{33}}. 
\end{eqnarray*}
Applying Lemma \ref{lemma:4.1}(a) with $u= \boldsymbol a\,\varphi$ we can easily
bound
\begin{equation}\label{eq:4.7}
|T_{31}| \le C h^3 \|\psi_h\|_0\| \varphi\|_2,
\end{equation}
while Lemma \ref{lemma:4.1}(b) applied to $u=\boldsymbol a(\varphi-F_h\varphi)$
yields
\begin{equation}\label{eq:4.8}
|T_{32}| \le h \|\psi_h\|_0 (\| \boldsymbol a(\varphi-F_h\varphi)\|_0+
h\|\boldsymbol a(\varphi-F_h\varphi)\|_1) \le C h^3 \|\psi_h\|_0 \| \varphi\|_2.
\end{equation}
Finally, we apply Lemma \ref{lemma:4.1}(b) again, using the bound for
$R_h\varphi$ provided by \eqref{eq:4.31}, which yields
\begin{equation}\label{eq:4.9}
|T_{33}|\le \|\psi_h\|_0 (\|R_h\varphi\|_0+\pi h\| R_h \varphi\|_1) \le C h^2
\|\psi_h\|_0\|\varphi\|_1.
\end{equation}
Inequalities \eqref{eq:4.7}-\eqref{eq:4.9} imply that
\begin{equation}\label{eq:4.10}
|T_3+h \underline C_1(\varepsilon) (\psi_h,\boldsymbol a \varphi)| \le C h^2
\|\psi_h\|_0 \|\varphi\|_2.
\end{equation}
Carrying \eqref{eq:4.5}, \eqref{eq:4.6}, and \eqref{eq:4.10} to \eqref{eq:4.4}
the result follows.
\end{proof}

\begin{proposition}\label{prop:4.6}
Let $\alpha$ be the constant in \eqref{eq:1.0b}. Then
\[
| w_h(\psi_h,\varphi-D_h\varphi)+ h\,\alpha\,\underline C_1(\varepsilon)
(\psi_h',\varphi')|\le C h^2 \|\psi_h\|_0\|\varphi\|_3 \qquad \forall \psi_h \in
S_{h,\varepsilon}.
\]
\end{proposition}

\begin{proof}
Using \eqref{eq:4.3}, we can write
\begin{equation}\label{eq:4.11}
w_h(\psi_h,\varphi-D_h\varphi) = \underbrace{\{\psi_h',\mathrm
V_1(\varphi'-Q_{h,1/2}^{-1}F_h
\varphi')\}}_{=:T_4}+\underbrace{\{Q_{h,\varepsilon}^{-1}\psi_h,
\mathrm V_2Q_h^{-1}(\varphi-D_h\varphi)\}}_{=:T_5}.
\end{equation}
To estimate $T_4$ we will use Proposition \ref{prop:4.3} applied to $\mathrm
V_1$ (see Proposition \ref{prop:2.2}). An easy computation shows that
\begin{equation}\label{eq:4.12}
\{\psi_h',\underline C_\ell(\punto/h+1/2)\,u\}=\underline C_\ell(\varepsilon)
(\psi_h',u) \qquad \forall \psi_h \in S_{h,\varepsilon}.
\end{equation}
Let $\mathrm L_1$ and $\mathrm L_2$ be the differential operators associated to
the expansion of $\mathrm V_1$ in Proposition \ref{prop:4.3} (note that $\mathrm
L_1 u=\alpha u$). By \eqref{eq:4.12} and Proposition \ref{prop:4.3} it follows
that
\begin{eqnarray}\nonumber
T_4 &=& - h \underline C_1(\varepsilon)(\psi_h,\mathrm L_1\varphi')\\
& & + \underbrace{h\underline C_1(\varepsilon) (\psi_h',\mathrm
L_1(\varphi'-F_h\varphi'))}_{=:T_{41}}-\underbrace{h^2\underline
C_2(\varepsilon) (\psi_h',\mathrm L_2F_h
\varphi')}_{=:T_{42}}+\underbrace{(\psi_h',T_h\varphi')}_{=:T_{43}}.
\label{eq:4.13}
\end{eqnarray}
Using \eqref{eq:4.2} we can easily bound
\begin{equation}\label{eq:4.14}
|T_{41}|=|h \underline C_1(\varepsilon) (\psi_h,\mathrm D \mathrm
L_1(\varphi'-F_h\varphi')|\le C h \|\psi_h\|_0 \|\varphi'-F_h\varphi'\|_1 \le C'
h^2 \|\psi_h\|_0\|\varphi\|_3
\end{equation}
and
\begin{equation}\label{eq:4.15}
|T_{42}|=|h^2 \underline C_2(\varepsilon) (\psi_h,\mathrm D \mathrm L_2
F_h\varphi')|\le C h^2 \|\psi_h\|_0 \|\varphi\|_3.
\end{equation}
Similarly, using the bound for $T_h$ given by \eqref{eq:4.32}, we estimate
\begin{equation}\label{eq:4.16}
|T_{43}| \le \|\psi_h\|_0 \| \mathrm D T_h \varphi'\|_0 \le C h^2
\|\psi_h\|_0\|\varphi\|_3.
\end{equation}
Taking \eqref{eq:4.14}-\eqref{eq:4.16} to \eqref{eq:4.13} we have proved that
\begin{equation}\label{eq:4.17}
| T_4+ h\alpha\, \underline C_1(\varepsilon) (\psi_h',\varphi')|\le C
h^2\|\psi_h\|_0 \|\varphi\|_3. 
\end{equation}
We next estimate the term $T_5$ in \eqref{eq:4.11}. Note that
\begin{equation}\label{eq:4.50}
Q_h^{-1} (\underline B_1(\punto/h+1/2) u)=\underline B_1(1/2) Q_h^{-1} u=0,
\end{equation}
by the coincidence of $\underline B_1$ with the Bernoulli polynomial of first
degree in $(0,1)$. Therefore, using Proposition \ref{prop:4.4} and Lemma
\ref{lemma:4.2}(b) it follows that
\begin{equation}\label{eq:4.18}
|T_5| =|\{ Q_{h,\varepsilon}^{-1} \psi_h,\mathrm V_2 Q_h^{-1} E_h\varphi\}| \le
C_\varepsilon \|\psi_h\|_0 (\|E_h\varphi\|_0+ h \|E_h \varphi \|_1) \le
C_\varepsilon' h^2 \|\psi_h\|_0 \|\varphi\|_2.
\end{equation}
The collection of \eqref{eq:4.11}, \eqref{eq:4.17} and \eqref{eq:4.18} proves
the result.
\end{proof}

\begin{corollary}\label{cor:4.7}
Let $\varphi_h$ be the solution of \eqref{eq:2.7} with right-hand side
$g=\mathrm W \varphi$. Then
\begin{equation}
| w_h(\psi_h,\varphi_h-D_h\varphi)+ h  \underline C_1(\varepsilon)
(\psi_h,\boldsymbol a\,\varphi-\alpha \varphi'')| \le C h^2\|\psi_h\|_0
\|\varphi\|_3 \qquad \forall \psi_h \in S_{h,\varepsilon},
\end{equation}
and
\begin{equation}\label{eq:4.41}
| w_h(\psi_h,\varphi_h-D_h\varphi)+ h  \underline C_1(\varepsilon)
(Q_{h,\varepsilon}^{-1}\psi_h,\boldsymbol a\,\varphi-\alpha \varphi'')| \le C
h^2\|\psi_h\|_0 \|\varphi\|_3 \qquad \forall \psi_h \in S_{h,\varepsilon},
\end{equation}
where $a(s)=A_2(s,s)$ and $\alpha$ is the constant in \eqref{eq:1.0b}.
\end{corollary}

\begin{proof} The first bound is a straightforward consequence of Propositions
\ref{prop:4.5} and \ref{prop:4.6}. The bound \eqref{eq:4.41}  can be derived
from the first and \eqref{eq:4.40}, although it has already been implicitly
given in the proofs above.
\end{proof}

\begin{proposition}[Zero order asymptotics]
The following reduced estimate holds:
\begin{equation}\label{eq:4.42}
| w_h(\psi_h,\varphi_h-D_h\varphi)|\le C h \|\psi_h\|_0 \|\varphi\|_2 \qquad
\forall \psi_h \in S_{h,\varepsilon}.
\end{equation}
\end{proposition}

\begin{proof} If we go back to the notation of the proofs of Propositions
\ref{prop:4.5} and \ref{prop:4.6}, it is clear from \eqref{eq:4.6},
\eqref{eq:4.10}, and \eqref{eq:4.18} that
\[
|T_2|+|T_3|+|T_5|\le C h \|\psi_h\|_0 \|\varphi\|_2 \qquad \forall \psi_h \in
S_{h,\varepsilon}^{-1}.
\]
Using \eqref{eq:4.40} instead of Lemma \ref{lemma:4.1}(a), it is also simple to
bound
\[
|T_1| \le C h \|\psi_h\|_0 \|\varphi\|_2 \qquad \forall \psi_h \in
S_{h,\varepsilon}^{-1}.
\]
For the operator $T_h^\circ u:=\mathrm V u-\mathrm V Q_{h,1/2}^{-1} F_h u + h
\underline C_1(\punto/h+1/2) \mathrm L_1 F_h u$, we can bound
$\|T_h^\circ u\|_0+ h \|T_h^\circ u\|_1 \le C h^2 \|u\|_1$ \cite[Proposition
16]{CeDoSa:2002}. Using this bound instead of Proposition \ref{prop:4.3}, we can
prove that
\[
|T_4|\le C h \|\psi_h\|_0 \|\varphi\|_2 \qquad \forall \psi_h \in
S_{h,\varepsilon}^{-1}.
\]
This finishes the proof.
\end{proof}

In order to set up clearly the precise formulas of the second term in the
asymptotic expansion of $w_h(\psi_h,\varphi_h-D_h\varphi)$, we need to consider
the first two terms in the expansions of $\mathrm V_1$ and $\mathrm V_2$ given
by Proposition \ref{prop:2.2}:
\begin{eqnarray*}
\mathrm V_2 &=& \boldsymbol a \mathrm H\mathrm D_{-1}+\boldsymbol b\mathrm
H\mathrm D_{-2} +\mathrm K_3, \qquad \mathrm K_3\in \mathcal E(-3),\\
\mathrm V_1 &=& \alpha\mathrm H\mathrm D_{-1}+\boldsymbol c\mathrm H\mathrm
D_{-2} +\mathrm J_3, \qquad \mathrm J_3\in \mathcal E(-3).
\end{eqnarray*}

\begin{proposition}[Second order asymptotics]\label{prop:A.4}
Let $\varphi_h$ be the solution of \eqref{eq:2.7} with right-hand side
$g=\mathrm W \varphi$. Let
\begin{eqnarray*}
\mathrm P_1^\varepsilon &:=& \underline C_1(\varepsilon) \Big(\alpha \mathrm D^2
-\boldsymbol a\Big),\\
\mathrm P_2^\varepsilon &:=& \underline C_2(\varepsilon) \Big( \mathrm D (\alpha
\mathrm D-\boldsymbol c)\mathrm D-(\boldsymbol a\mathrm D-\boldsymbol
b)\Big)+\smallfrac1{24} (\mathrm D^3\mathrm V_1\mathrm D+\mathrm V_2\mathrm
D^2).
\end{eqnarray*}
Then
\[
\Big| w_h(\psi_h,\varphi_h-D_h\varphi)-\sum_{\ell=1}^2 h^\ell  
(Q_{h,\varepsilon}^{-1}\psi_h,\mathrm P_\ell^\varepsilon \varphi)\Big| \le C
h^3\|\psi_h\|_0 \|\varphi\|_4 \qquad \forall \psi_h \in S_{h,\varepsilon},
\]
\end{proposition}

\begin{proof} See Appendix \ref{sec:A}.\end{proof}

\section{Convergence theorems}\label{sec:5}

The results of Sections \ref{sec:3} and \ref{sec:4} give a first simple $H^0$
estimate of the convergence of the method, showing that when
$\varepsilon=\pm1/6$, the solution superconverges to the projection
$D_h\varphi$. We first recall that \cite[Formula (5)]{DoSa:2001}
\begin{equation}\label{eq:5.1}
\|D_h\varphi-\varphi\|_s\le C h^{r-s}\| \varphi\|_r, \qquad s\le r\le 1, \qquad
s<1/2.
\end{equation}

\begin{theorem}\label{the:5.1}
Let $\varphi_h$ be the solution of \eqref{eq:2.7} with right-hand side
$g=\mathrm W \varphi$ and $\varepsilon\not\in \frac12\mathbb Z$. Then
\begin{equation}\label{eq:5.2}
\| \varphi_h - \varphi\|_0 \le C_\varepsilon h \|\varphi\|_2.
\end{equation}
Moreover,
\begin{equation}\label{eq:5.3}
\| \varphi_h-D_h\varphi\|_0\le C h^2 \|\varphi\|_3 \qquad \mbox{if
$\varepsilon\in \{-1/6,1/6\}$}.
\end{equation}
\end{theorem}

\begin{proof} Using Theorem \ref{the:3.1}, and \eqref{eq:4.42}, we can  prove
that
\[
\beta_\varepsilon \| \varphi_h-D_h\varphi\|_0 \le \sup_{0\neq \psi_h \in
S_{h,\varepsilon}}\frac{|w_h(\psi_h,\varphi_h-D_h\varphi)|}{\|\psi_h\|} \le C
h\|\varphi\|_2 .
\]
Applying \eqref{eq:5.1}, this proves \eqref{eq:5.2}. The superconvergence bound
\eqref{eq:5.3} follows from Corollary \ref{cor:4.7}  (note that $\underline
C_1(\pm 1/6)=0$) and Theorem \ref{the:3.1}.
\end{proof}

The superconvergence estimate can be first exploited with a postprocessing of
the solution: given $v$ smooth enough we approximate
\[
\int_0^1 \varphi(t) v(t) \mathrm dt \approx h \sum_{j=0}^{N-1} \varphi_j\,
v(t_j)=(Q_h^{-1}\varphi_h,v).
\]
This includes the fully discrete double layer potential \eqref{eq:1.17} to
approximate \eqref{eq:1.13}. 

\begin{corollary} \label{cor:5.2}
Let $\varphi_h$ be the solution of \eqref{eq:2.7} with right-hand side
$g=\mathrm W \varphi$, and $\varepsilon\in \{-1/6,1/6\}$. Then
\[
|(Q_h^{-1}\varphi_h,v)-(\varphi,v)| \le C h^2 \|\varphi\|_3 \| v\|_2 \qquad
\forall v \in H^2.
\]
\end{corollary}

\begin{proof} Using Lemma \ref{lemma:4.1}(a) (with $\varepsilon=0$), we can
easily bound
\begin{eqnarray*}
|(Q_h^{-1}\varphi_h,v)-(\varphi,v)|&\le&
|(Q_h^{-1}\varphi_h-\varphi_h,v)|+|(\varphi_h-D_h\varphi,
v)|+|(D_h\varphi-\varphi,v)|\\
&\le & C_1 h^2\|\varphi_h\|_0 
\|v\|_2+\|\varphi_h-D_h\varphi\|_0\|v\|_0+\|D_h\varphi-\varphi\|_{-1}\|v\|_1 \\
&\le & C_2 h^2
\Big(\|\varphi\|_2\|v\|_2+\|\varphi\|_3\|v\|_0+\|\varphi\|_1\|v\|_1\Big),
\end{eqnarray*}
by \eqref{eq:5.2}, \eqref{eq:5.3}, and \eqref{eq:5.1}.
\end{proof}

We next introduce the interpolation operator
\[
I_h u:= \sum_{j=0}^{N-1} u(t_j)\chi_j.
\]
The Sobolev embedding theorem \cite[Lemma 5.3.2]{SaVa:2002} and Proposition
\ref{prop:4.4} show that 
\[
\max_j |u(t_j)-(D_hu)(t_j)-h \underline B_1(t_j/h+1/2) (F_h u')(t_j)|\le \|
E_hu\|_{L^\infty}\le C\| E_hu\|_1\le C h \|u\|_2.
\]
However, $\underline B_1(t_j/h+1/2)=\underline B_1(1/2)=0$ and therefore
\begin{equation}\label{eq:5.4}
\| I_h u-D_hu\|_{L^\infty}\le C h \| u\|_2\qquad \mbox{and}\qquad \|
D_hu\|_{L^\infty} \le C \|u\|_2.
\end{equation}

\begin{theorem}\label{the:5.3}
Let $\varphi_h$ be the solution of \eqref{eq:2.7} with right-hand side
$g=\mathrm W \varphi$ and $\varepsilon\not\in \frac12\mathbb Z$. Then
\[
\max_j |\varphi_j-\varphi(t_j)|=\| \varphi_h-I_h \varphi\|_{L^\infty} \le C h
\|\varphi\|_3.
\]
\end{theorem}

\begin{proof}
We rely on the first order asymptotic formula of Corollary \ref{cor:4.7}. Let
$C_h$ be the solution operator associated to \eqref{eq:2.7}, namely
$C_h\varphi=\varphi_h$. Let $\xi:=\underline C_1(\varepsilon)\,\mathrm W^{-1}
(\alpha \varphi''-\boldsymbol a \varphi).$ Then \eqref{eq:4.41} shows that
\[
|w_h(\psi_h,\varphi_h-D_h\varphi)-h (Q_{h,\varepsilon}^{-1}\psi_h,\mathrm W
\xi)|\le C h^2 \|\psi_h\|_0\|\varphi\|_3 \qquad \forall \psi_h \in
S_{h,\varepsilon}^{-1},
\]
which can also be written as
\[
| w_h(\psi_h,C_h \varphi-D_h\varphi-h C_h \xi)|\le C h^2
\|\psi_h\|_0\|\varphi\|_3 \qquad \forall \psi_h \in S_{h,\varepsilon}^{-1}.
\]
By the inf-sup condition in Theorem \ref{the:3.1}, it follows that
\[
\| C_h\varphi-D_h\varphi-h C_h\xi\|_0\le C h^2 \|\varphi\|_3,
\]
and therefore, by \eqref{eq:5.2} applied to $\xi$,
\begin{eqnarray}\nonumber
\| C_h\varphi-D_h\varphi-h D_h\xi\|_0&\le & h\|C_h\xi-D_h\xi\|_0+C h^2
\|\varphi\|_3\le Ch^2 (\|\xi\|_2+\|\varphi\|_3)\\
&\le &C' h^2 \|\varphi\|_3,\label{eq:5.5}
\end{eqnarray}
since  $\|\xi\|_2\le C\|\varphi\|_3$
Note that for piecewise constant functions on a uniform grid of meshsize $h$ we
can estimate $\|\rho_h\|_{L^\infty}\le h^{-1/2}\|\rho_h\|_0$. 
Thus,
\begin{eqnarray*}
\| C_h \varphi-I_h\varphi\|_{L^\infty} &\le & h^{-1/2}\| C_h\varphi-D_h\varphi-h
D_h\xi\|_0+\|D_h\varphi-I_h\varphi\|_{L^\infty} + h\| D_h \xi\|_{L^ \infty}\\
&\le & C h^{3/2}\|\varphi\|_3+ C h \|\varphi\|_2+ C h \|\xi\|_2,
\end{eqnarray*}
by \eqref{eq:5.4}. This proves the result.
\end{proof}

\begin{theorem} \label{the:5.4}
Let $\varphi_h$ be the solution of \eqref{eq:2.7} with right-hand side
$g=\mathrm W \varphi$ and $\varepsilon\in \{-1/6,1/6\}$. Then
\[
\max_j |\varphi_j-\varphi(t_j)|=\| \varphi_h-I_h \varphi\|_{L^\infty} \le C h^2
\|\varphi\|_4.
\]
\end{theorem}

\begin{proof}
The proof of this estimate is very similar to that of Theorem \ref{the:5.3}. We
need to rely on the second order asymptotics of the error (Proposition
\ref{prop:A.4}) to reveal the first non-vanishing term in the asymptotic error
expansion when $\varepsilon \in \{-1/6,1/6\}$.In addition to this, using
Proposition \ref{prop:A.3} and the Sobolev imbedding theorem, it is easy to show
that the estimate \eqref{eq:5.4} can be improved to $\| I_h u-D_h u\|_{L^\infty}
\le C h^2 \| u\|_3.$
An inverse inequality, the stability estimate (Theorem \ref{the:3.1}) and
Theorem \ref{the:5.1}, can be used to show that
\[
\| \varphi_h - D_h\varphi-h^2 D_h \gamma\|_{L^\infty} \le C
h^{5/2}\|\varphi\|_4, \mbox{ where } \gamma:= \mathrm W^{-1} \mathrm
P_2^\varepsilon \varphi,
\] 
and $\mathrm P_2^\varepsilon\in \mathcal E(2)$ is given in Proposition
\ref{prop:A.4}.
All remaining details are omitted.
\end{proof}

\section{Numerical experiments}\label{sec:7}

We will now illustrate some of the previous convergence estimates with a simple
example. We take two ellipses, one centered at $(0,0)$ with semiaxes $1$ and
$2$, and a second one centered at $(4,5)$ with semiaxes $2$ and $1$. We look for
solutions of \eqref{eq:1.14} (radiating solutions of the Helmholtz equations) in
the exterior domain that lies outside both ellipses, with Neumann conditions on
the boundaries (see Section \ref{sec:1.3}). As exact solution we take $U(\mathbf
z):=H^{(1)}_0(k|\mathbf z-\mathbf z_0|)$, where $\mathbf z_0:=(0.1, 0.2)$ is a
point inside the first of the obstacles. We have taken $k=1$ in all examples.

\paragraph{Experiment \#1 (indirect method).} After parametrization of the
ellipses, a double layer potential \eqref{eq:1.13} is defined on each of the
curves. They are then used to set up a $2\times 2$ system of parametrized
boundary integral equations, with diagonal terms of the form \eqref{eq:1.1} and
integral operators with smooth kernels as off--diagonal terms. We solve the
system and plug the resulting densities in the fully discrete double layer
potentials \eqref{eq:1.17}. We compute the errors:
\[
e_h:=\max_{\mathbf z \in \mathrm O} |U(\mathbf z)-U_h(\mathbf z)|,
\quad\mbox{where}\quad \mathrm O:=\{(−4,− 4), (−5, − 5.5), (−6 ,− 7), (7, 7.6),
(−6.8, − 6)\},
\]
that is, we observe the difference of the exact and discrete solutions at five
external points. We expect $e_h=\mathcal O(h)$ (this follows from Theorem
\ref{the:5.1}) and $e_h=\mathcal O(h^2)$ when $\varepsilon\in \{-1/6,1/6\}$
(Corollary \ref{cor:5.2}). The results are shown in Table \ref{table:1}. To see
how the superconvergent values of $\varepsilon$ are reflected in the error, we
plot the error $e_h$ as a function of $\varepsilon$ for a fixed value of $N$ in
Figure \ref{figure:1}.

\begin{table}[htb]
\begin{center}
\begin{tabular}{rcc}\hline
 $N$&error&e.c.r\\[0.5ex] \hline
10&4.3005$E(-002)$&\\ 
20&1.9193$E(-002)$&1.1640\\ 
40&9.0917$E(-003)$&1.0779\\
80&4.4279$E(-003)$&1.0379\\
160&2.1852$E(-003)$&1.0189\\
320&1.0855$E(-003)$&1.0094\\
640&5.4097$E(-004)$&1.0047\\
\hline
\end{tabular}\qquad
\begin{tabular}{rcc}\hline
$N$&error&e.c.r\\[0.5ex] \hline
10&9.7262$E(-003)$&\\ 
20&2.5602$E(-003)$&1.8995\\ 
40&6.2157$E(-004)$&2.0645\\
80&1.5443$E(-004)$&2.0090\\
160&3.8588$E(-005)$&2.0007\\
320&9.6507$E(-006)$&1.9995\\
640&2.4135$E(-006)$&1.9995\\
\hline
\end{tabular}
\end{center}
\caption{Errors and estimated convergence rates for Experiment \#1. The variable
$N$ is the number of points on each of the curves. The leftmost table
corresponds to $\varepsilon=1/3$ (order one) and the rightmost table to
$\varepsilon=1/6$.}\label{table:1}
\end{table}  

\begin{figure}[htb]
\begin{center}
\includegraphics[width=10cm]{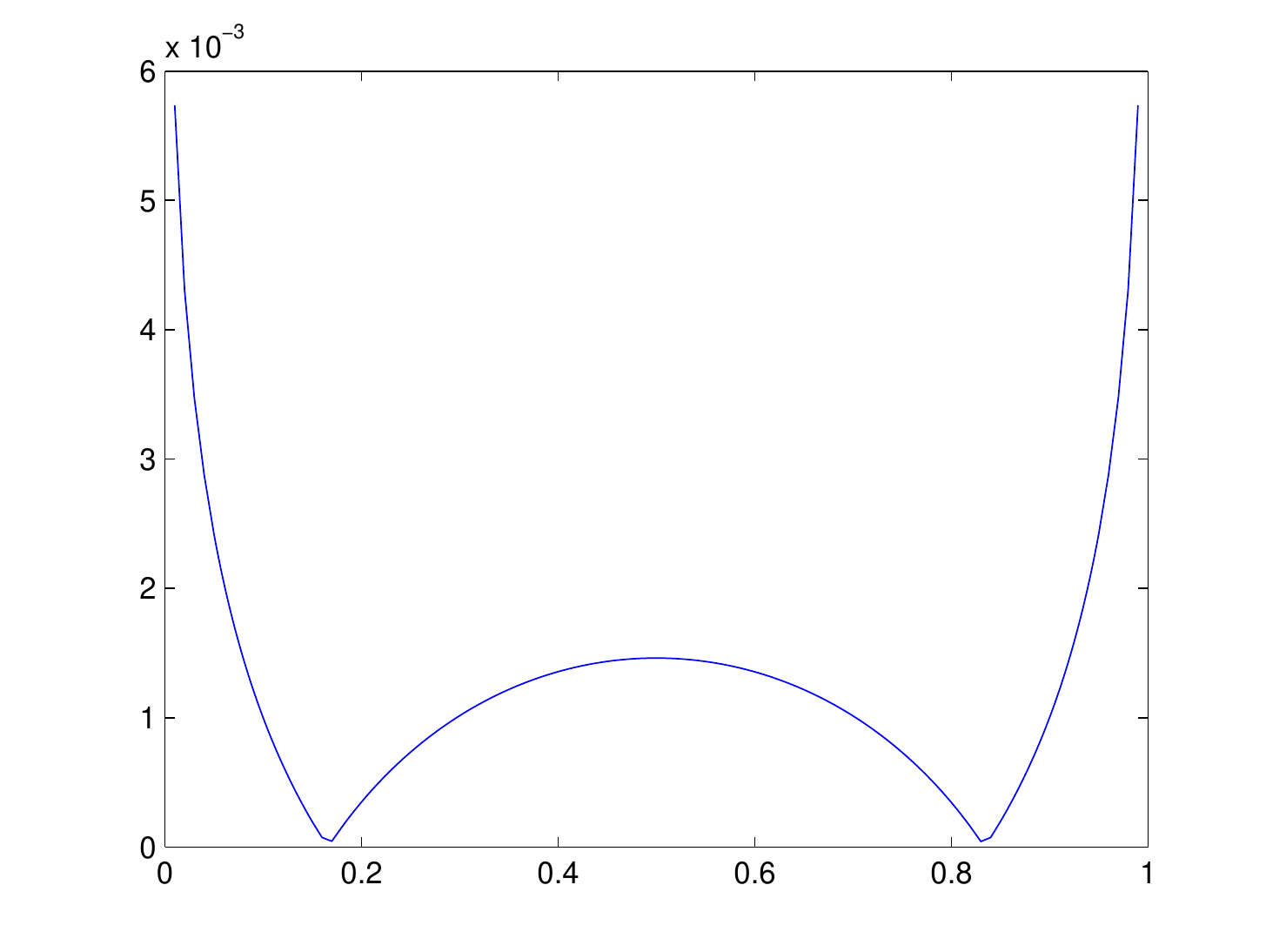}
\end{center}
\caption{Error as a function of $\varepsilon$ in Experiment \# 1. The
superconvergent methods can be clearly seen as kinks in the error graph
(corresponding to the first term in the asymptotic expansion of the error going
through a zero). The methods becomes unstable as $\varepsilon \to \mathbb Z$.
Although our analysis does not cover this case, it is clear from the graph that
$\varepsilon=1/2$ continues smoothly the error graph.}\label{figure:1}
\end{figure}

\paragraph{Experiment \#2 (Richardson extrapolation).} With the same geometric
configuration, exact solution, and numerical scheme as in the superconvergent
case ($\varepsilon=1/6$), we apply Richardson extrapolation to propose the
potential
\[
U_h^\star:=\smallfrac43 U_{h/2}-\smallfrac13 U_h,
\]
as an improved approximation of the solution. The result of Proposition
\ref{prop:A.4} points clearly to the existence of an asymptotic expansion of the
error, very much in the style of those obtained for operator equations of zero
or negative order in \cite{CeDoSa:2002}. The numerical result shown in Table
\ref{table:2} corresponding maximum errors $e_h:=\max_{\mathbf z \in \mathrm O}
|U(\mathbf z)-U_h(\mathbf z)|=\mathcal O(h^3).$

\begin{table}[htb]
\begin{center}
\begin{tabular}{rcc}\hline
$ N$&error&e.c.r\\[0.5ex] \hline
10&4.3437$E(-006)$&\\ 
20&7.4235$E(-008)$&5.8707\\ 
40&5.6231$E(-009)$&3.7227\\
80&6.6107$E(-010)$&3.0885\\
160&8.1033$E(-011)$&3.0282\\
320&1.0052$E(-011)$&3.0110\\
\hline
\end{tabular}
\end{center}
\caption{Errors and estimated convergence rates for Experiment \#2: Richardson
extrapolation applied to one of the surperconvergent methods
($\varepsilon=1/6$). Errors are computed in several external observation points.
The result reported with $N=20$ uses a grid of $20$ points as $h-$grid and a
refined grid of $40$ points as $h/2-$grid.}\label{table:2}
\end{table}

\paragraph{Experiment \#3 (direct method).} We now apply a direct boundary
integral equation method for the same exterior Neumann problem as in the
previous experiments. This leads to a $2\times 2$ system with the same matrix of
operators as in the previous formulation, but  the adjoint double layer operator
appears in the right-hand side of the system. This operator is simply
discretized with midpoint formulas on each of the intervals: see
\cite{DoRaSa:2008} for a similar treatment in systems related to the single
layer potential. With this formulation, the unknown is the parametrized form of
the trace of the exact solution $\varphi=U\circ\mathbf x$ and we can thus
compare $L^\infty$ errors (Theorems \ref{the:5.3} and \ref{the:5.4}). We measure
maximum absolute value of errors for $\varphi$ on the points $t_j$. The results
are reported in Table \ref{table:3}.

\begin{table}[htb]  
\begin{center}
\begin{tabular}{rcccc}\hline
 $N$&boundary error&e.c.r\\[0.5ex] \hline
10&4.8555$E(-001)$&\\ 
20&1.3426$E(-001)$&1.8546\\ 
40&5.4891$E(-002)$&1.2904\\
80&2.4641$E(-002)$&1.1555\\
160&1.1792$E(-002)$&1.0632\\
320&5.7677$E(-003)$&1.0318\\
640&2.8527$E(-003)$&1.0157\\
\hline
\end{tabular}
\qquad
\begin{tabular}{rcc}\hline
 $N$&boundary error&e.c.r\\[0.5ex] \hline
10&2.0406$E(-001)$&\\ 
20&3.5629$E(-002)$&2.5179\\ 
40&8.6392$E(-003)$&2.0441\\
80&2.1497$E(-003)$&2.0068\\
160&5.3603$E(-004)$&2.0037\\
320&1.3385$E(-004)$&2.0017\\
640&3.3444$E(-005)$&2.0008\\
\hline
\end{tabular}\end{center}
\caption{Errors and estimated convergence rates for Experiment \#3. The leftmost
table corresponds to $\varepsilon=1/3$ (order one) and the rightmost table to
$\varepsilon=1/6$. The table shows errors
$\|\varphi_h-I_h\varphi\|_{L^\infty}$.}\label{table:3}
\end{table}  

\paragraph{Experiment \#4 (condition numbers).} In this final experiment, we
pick the matrix of the previous examples and compute its spectral condition
number. We then show how a Calder\'on preconditioner based on premultiplying the
matrix $\mathrm W_{ij}$ by a matrix \cite{CeDoSa:2002, DoRaSa:2008}
\begin{equation}\label{eq:B.1}
\mathrm V_{ij} = H^{(1)}_0 (k|\mathbf x(t_i)-\mathbf x(t_{j+\varepsilon})|)
\end{equation}
reduces the condition number of the resulting system to what is basically a
constant $h-$independent condition number.

\begin{table}[htb]
\begin{center}
\begin{tabular}{rcr}\hline
 N&cond VW&cond W\\[0.5ex] \hline
10&6.9548&5.7212\\ 
20&6.5994&11.7992\\ 
40&6.5349&23.7403\\
80&6.5196&47.5489\\
160&6.5159&95.1320\\
320&6.5150&190.2811\\
640&6.5148&380.5709\\
\hline
\end{tabular}
\end{center}  
\caption{Condition numbers for the matrix $\mathrm W$ 
of Experiments \# 1  and \# 2 and for the matrix $\mathrm V\mathrm W$, with
$\mathrm V$ given by \eqref{eq:B.1}.}\label{table:4}
\end{table}

\appendix

\section{Second order asymptotics}\label{sec:A}

This section contains the proof of Proposition \ref{prop:A.4}. Note that this
result is required for the proof of $L^\infty$ convergence of the
superconvergent methods. In order to prove Proposition \ref{prop:A.4} we have to
go one term further in the different asymptotic expansions that were used in the
proofs of Propositions \ref{prop:4.5} and \ref{prop:4.6}.
\subsection{Technical background}

\begin{lemma}\label{lemma:A.1} 
There exists $C$ such that for all $\varepsilon$ and $h$
\[
\Big| \{ Q_{h,\varepsilon}^{-1}\psi_h, u\}-(\psi_h,u)+\smallfrac{h^2}{24} \{
Q_{h,\varepsilon}^{-1}\psi_h,u''\}\Big| \le C h^3\| \psi_h\|_0\| u\|_3 \qquad
\forall \psi_h\in S_{h,\varepsilon},\, u \in H^3.
\]
\end{lemma}

\begin{proof}
It is based on the same ideas as the proof of Lemma \ref{lemma:4.1}, using the
inequality
\[
\Big|\int_{c-\frac{h}2}^{c+\frac{h}2} u(t)\mathrm d t-hu(c)-\frac{h^3}{24}
u''(c)\Big|\le C h^3 \int_{c-\frac{h}2}^{c+\frac{h}2} |u^{(3)}(t)|\mathrm d t,
\]
as starting point.
\end{proof}

\begin{proposition}\label{prop:A.2}
Let $a_1,a_2,a_3\in \mathcal D$ and consider an operator
\[
\mathrm V:=\boldsymbol a_1\mathrm H\mathrm D_{-1}+\boldsymbol a_2\mathrm
H\mathrm D_{-2}+\boldsymbol a_3\mathrm H\mathrm D_{-3}+\mathrm K_4 \qquad
\mbox{where $\mathrm K_4\in \mathcal E(-4)$}.
\]
Let then $\mathrm L_1:=\boldsymbol a_1$, $\mathrm L_2:=\boldsymbol a_1\mathrm
D-\boldsymbol a_2$, $\mathrm L_3:=\boldsymbol a_1\mathrm D^2-2\boldsymbol
a_2\mathrm D+\boldsymbol a_3,$
 and consider the operators
\begin{eqnarray*}
R_h^{\#} u &:=& \mathrm V u-\mathrm V Q_h^{-1} F_h u + \sum_{\ell=1}^2 h^\ell
\underline C_\ell(\punto/h)\mathrm L_\ell F_h u,\\
T_h^{\#} u &:=& \mathrm V u-\mathrm V Q_{h,1/2}^{-1} F_h u +\sum_{\ell=1}^3
h^\ell \underline C_\ell(\punto/h+1/2)\mathrm L_\ell F_h u . \end{eqnarray*}
Then
\begin{eqnarray}\label{eq:A.1}
\|R_h^{\#} u\|_0+h\| R_h^{\#} u\|_1&\le & C h^3 \|u\|_2 \qquad \forall u \in
H^2,\\
\|T_h^{\#} u\|_0+h\| T_h^{\#} u\|_1&\le & C h^4 \|u\|_3 \qquad \forall u \in
H^3.
\label{eq:A.2}
\end{eqnarray}
\end{proposition}

\begin{proof}
It is a direct consequence of \cite[Proposition 16]{CeDoSa:2002}.
\end{proof}

\begin{proposition}\label{prop:A.3}
Let $E_h^{\#} u:= u -D_h u +\sum_{\ell=1}^2 h^\ell \underline
B_\ell(\punto/h+1/2)F_h u^{(\ell)}.$
Then
\[
\| E_h^{\#}u\|_0+ h\|E_h^{\#} u\|_1\le C h^3 \|u\|_3 \qquad \forall u \in H^3.
\]
\end{proposition}

\begin{proof} See \cite[Proposition 1]{DoSa:2001} and the proof of
\cite[Proposition 16]{CeDoSa:2002}.
\end{proof}

\subsection{Proof of Proposition \ref{prop:A.4}}

Following \eqref{eq:4.4} and \eqref{eq:4.11}, we consider the decomposition of
the consistency error in five terms 
\begin{equation}\label{eq:A.3}
w_h(\psi_h,\varphi_h-D_h\varphi)=w_h(\psi_h,\varphi_h-\varphi)+w_h(\psi_h,
\varphi-D_h\varphi)= (T_1+T_2+T_3)+(T_4+T_5).
\end{equation}
To bound $T_1$ we use Lemma \ref{lemma:A.1} with $u=\mathrm D\mathrm V_1\mathrm
D$:
\begin{equation}\label{eq:A.4}
\Big| T_1-\smallfrac{h^2}{24}\{ Q_{h,\varepsilon}^{-1}\psi_h,\mathrm D^3\mathrm
V_1\mathrm D\varphi\}\Big|\le C h^3 \|\psi_h\|_0\|\varphi\|_4.
\end{equation}
Proceeding as in \eqref{eq:4.6} we can bound
\begin{equation}\label{eq:A.5}
|T_2| \le C h^3 \|\psi_h\|_0\|\varphi\|_3.
\end{equation}
To expand $T_3$ we use Proposition \ref{prop:A.2} applied to $\mathrm V_2$ and
\eqref{eq:4.2} to obtain
\begin{eqnarray*}
T_3 &\!\!=\!\!& - h\underline C_1(\varepsilon) (Q_{h,\varepsilon}^{-1}
\psi_h,\boldsymbol a\varphi)  - h^2\underline C_2(\varepsilon)
(Q_{h,\varepsilon}^{-1} \psi_h,(\boldsymbol a\mathrm D -\boldsymbol b)\varphi)\\
& &  + \underbrace{h\underline C_1(\varepsilon) (Q_{h,\varepsilon}^{-1}
\psi_h,\boldsymbol a(\varphi-F_h\varphi))}_{T_{3a}}  + \underbrace{h^2\underline
C_2(\varepsilon) (Q_{h,\varepsilon}^{-1} \psi_h,(\boldsymbol a\mathrm D
-\boldsymbol b)(\varphi-F_h\varphi))}_{T_{3b}}\\
& & +\underbrace{ (Q_{h,\varepsilon}^{-1}\psi_h,R_h^{\#}\varphi)}_{T_{3c}}.
\end{eqnarray*}
We now use Lemma \ref{lemma:4.1}(b) and \eqref{eq:4.2} to bound $|T_{3a}|
+|T_{3b}|\le C h^3 \|\psi_h\|_0\|\varphi\|_2,$
as well as Lemma \ref{lemma:4.1}(b) and \eqref{eq:A.1} to bound
$|T_{3c}| \le C h^3 \|\psi_h\|_0\|\varphi\|_2.$
Therefore
\begin{equation}\label{eq:A.6}
|T_3 + h\underline C_1(\varepsilon) (Q_{h,\varepsilon}^{-1} \psi_h,\boldsymbol
a\varphi)  + h^2\underline C_2(\varepsilon) (Q_{h,\varepsilon}^{-1}
\psi_h,(\boldsymbol a\mathrm D -\boldsymbol b)\varphi)|\le C
h^3\|\psi_h\|_0\|\varphi\|_2.
\end{equation}
To expand $T_4$ we use Proposition \ref{prop:A.2} applied to $\mathrm V_1$. Note
that $\mathrm L_1=\alpha\mathrm I$, $\mathrm L_2=\alpha\mathrm D-\boldsymbol c$
and $\mathrm L_3\in \mathcal E(2)$. Because of \eqref{eq:4.12}, we can write
\begin{eqnarray*}
T_4 &=& \sum_{\ell=1}^2 h^\ell \underline C_\ell(\varepsilon) (
Q_{h,\varepsilon}^{-1}\psi_h,\mathrm D \mathrm L_\ell \mathrm D\varphi)\\
& & +\sum_{\ell=1}^2  \underline
C_\ell(\varepsilon)\Big(\underbrace{h^\ell(\psi_h-Q_{h,\varepsilon}^{-1}\psi_h,
\mathrm D \mathrm L_\ell\mathrm D \varphi)}_{T_{4a}^\ell}+\underbrace{h^\ell
(\psi_h',\mathrm L_\ell (\varphi'-F_h\varphi'))}_{T_{4b}^\ell}\Big)\\
& & - \underbrace{ h^3 \underline C_3(\varepsilon) (\psi_h',\mathrm L_3
F_h\varphi')}_{T_{4c}}+\underbrace{(\psi_h',T_h^{\#}\varphi')}_{T_{4d}}.
\end{eqnarray*} 
Using Lemma \ref{lemma:4.1}(a) and \eqref{eq:4.40} we can bound
\[
|T_{4a}^1|+|T_{4a}^2| \le C h^3 \|\psi_h\|_0 \left( \|\mathrm D \mathrm
L_1\mathrm D\varphi\|_2+\|\mathrm D\mathrm L_2\mathrm D\varphi\|_1\right)\le C'
h^3\|\psi_h\|_0 \|\varphi\|_4.
\] 
By \eqref{eq:4.2} (and using the commutation $\mathrm D F_h=F_h\mathrm D$ to
simplify some expressions) we next bound
\begin{eqnarray*}
|T_{4b}^1|+|T_{4b}^2| &\le& \|\psi_h\|_0 \Big( h \|\mathrm D \mathrm L_1 \mathrm
D (\varphi-F_h\varphi)\|_0+ h^2  \|\mathrm D \mathrm L_2 \mathrm D
(\varphi-F_h\varphi)\|_0 \Big) \\
&\le & C \|\psi_h\|_0 \Big( h \| \varphi-F_h\varphi\|_2+h^2\|
\varphi-F_h\varphi\|_3\Big)\le C' h^3 \|\psi_h\|_0\|\varphi\|_4.
\end{eqnarray*}
Similarly
\[
|T_{4c}|\le h^3|\underline C_3(\varepsilon)| \|\psi_h\|_0 \| \mathrm D \mathrm
L_3\mathrm D F_h\varphi\|_0\le C\, h^3 \|\psi_h\|_0\|\varphi\|_4.
\]
Finally, by \eqref{eq:A.2}
\[
| T_{4d}| \le \|\psi_h\|_0 \| \mathrm D T_h^{\#} \varphi'\|_0 \le C h^3
\|\psi_h\|_0 \|\varphi\|_4.
\]
Collecting all these bounds we have just proved that
\begin{equation}\label{eq:A.7}
\Big|T_4 - \sum_{\ell=1}^2 h^\ell \underline C_\ell(\varepsilon) (
Q_{h,\varepsilon}^{-1}\psi_h,\mathrm D \mathrm L_\ell \mathrm D\varphi)
\Big|\le C h^3\|\psi_h\|_0\|\varphi\|_4.
\end{equation}
We are only left to deal with $T_5$. Using Proposition \ref{prop:A.3}, the
argument in \eqref{eq:4.50}, and the fact that $\underline B_2(1/2)=-1/24$, we
can write
\begin{eqnarray*}
Q_h^{-1}(\varphi-D_h\varphi)&=& -h \underline B_1(1/2) Q_h^{-1} F_h \varphi'-h^2
\underline B_2(1/2) Q_h^{-1} F_h\varphi''+Q_h^{-1} E_h^{\#}\varphi\\
&=& \frac{h^2}{24} Q_h^{-1} F_h \varphi''+Q_h^{-1} E_h^{\#}\varphi.
\end{eqnarray*}
Therefore,
\begin{eqnarray*}
T_5 &=& \frac{h^2}{24} \{ Q_{h,\varepsilon}^{-1}\psi_h,\mathrm V_2Q_h^{-1} F_h
\varphi''\} + \{ Q_{h,\varepsilon}^{-1}\psi_h,\mathrm V_2Q_h^{-1}
E_h^{\#}\varphi\}\\
&=& \frac{h^2}{24} ( Q_{h,\varepsilon}^{-1}\psi_h,\mathrm V_2 \varphi'')\\
& & +\underbrace{\frac{h^3}{24} \underline C_1(\varepsilon) (
Q_{h,\varepsilon}^{-1}\psi_h,\boldsymbol a
F_h\varphi'')}_{T_{5a}}-\underbrace{\frac{h^2}{24}
(Q_{h,\varepsilon}^{-1}\psi_h,R_h\varphi'')}_{T_{5b}}
+\underbrace{\{ Q_{h,\varepsilon}^{-1}\psi_h,\mathrm V_2Q_h^{-1}
E_h^{\#}\varphi\} }_{T_{5c}},
\end{eqnarray*} 
where we have applied Proposition \ref{prop:4.3}. By Lemma \ref{lemma:4.1}(b)
and \eqref{eq:4.31} we can bound
$|T_{5a}|+|T_{5b}|\le C h^3\|\psi_h\|_0\|\varphi\|_3, $
while by Lemma \ref{lemma:4.2} and Proposition \ref{prop:A.3}, we can bound 
$
|T_{5c}|\le C h^3\|\psi_h\|_0\|\varphi\|_3.
$
Therefore
\begin{equation}\label{eq:A.8}
\Big| T_5-\frac{h^2}{24} ( Q_{h,\varepsilon}^{-1}\psi_h,\mathrm V_2
\varphi'')\Big| \le C h^3\|\psi_h\|_0\|\varphi\|_3.
\end{equation}
The result is the combination of \eqref{eq:A.3}-\eqref{eq:A.8}.

\bibliographystyle{abbrv}
\bibliography{referencesW}

\begin{thebibliography}{10}

\bibitem{Arnold:1983}
D.~N. Arnold.
\newblock A spline-trigonometric {G}alerkin method and an exponentially
  convergent boundary integral method.
\newblock {\em Math. Comp.}, 41(164):383--397, 1983.

\bibitem{CeDoSa:2002}
R.~Celorrio, V.~Dom{\'i}nguez, and F.~J. Sayas.
\newblock Periodic {D}irac delta distributions in the boundary element method.
\newblock {\em Adv. Comput. Math.}, 17(3):211--236, 2002.

\bibitem{DoLuSa:2012}
V.~Dom{\'{\i}}nguez, S.~L. Lu, and F.-J. Sayas.
\newblock Fully discrete {C}alder\'on {C}alculus for the two dimensional
  {H}elmholtz equation.
\newblock In preparation.

\bibitem{DoRaSa:2008}
V.~Dom{\'i}nguez, M.-L. Rap{\'u}n, and F.-J. Sayas.
\newblock Dirac delta methods for {H}elmholtz transmission problems.
\newblock {\em Adv. Comput. Math.}, 28(2):119--139, 2008.

\bibitem{DoSa:2001}
V.~Dom{\'{\i}}nguez and F.-J. Sayas.
\newblock Full asymptotics of spline {P}etrov-{G}alerkin methods for some
  periodic pseudodifferential equations.
\newblock {\em Adv. Comput. Math.}, 14(1):75--101, 2001.

\bibitem{DoSa:2001a}
V.~Dom{\'{\i}}nguez and F.-J. Sayas.
\newblock Local expansions of periodic spline interpolation with some
  applications.
\newblock {\em Math. Nachr.}, 227:43--62, 2001.

\bibitem{HsWe:2008}
G.~C. Hsiao and W.~L. Wendland.
\newblock {\em Boundary integral equations}, volume 164 of {\em Applied
  Mathematical Sciences}.
\newblock Springer-Verlag, Berlin, 2008.

\bibitem{McLean:2000}
W.~McLean.
\newblock {\em Strongly elliptic systems and boundary integral equations}.
\newblock Cambridge University Press, Cambridge, 2000.

\bibitem{saSc:1993}
J.~Saranen and L.~Schroderus.
\newblock Quadrature methods for strongly elliptic equations of negative order
  on smooth closed curves.
\newblock {\em SIAM J. Numer. Anal.}, 30(6):1769--1795, 1993.

\bibitem{SaSl:1992}
J.~Saranen and I.~H. Sloan.
\newblock Quadrature methods for logarithmic-kernel integral equations on
  closed curves.
\newblock {\em IMA J. Numer. Anal.}, 12(2):167--187, 1992.

\bibitem{SaVa:2002}
J.~Saranen and G.~Vainikko.
\newblock {\em Periodic Integral and Pseudodifferential Equations with
  Numerical Approximation}.
\newblock Springer Monographs in Mathematics. Springer-Verlag, Berlin, 2002.

\bibitem{SauSch:2011}
S.~A. Sauter and C.~Schwab.
\newblock {\em Boundary element methods}, volume~39 of {\em Springer Series in
  Computational Mathematics}.
\newblock Springer-Verlag, Berlin, 2011.

\bibitem{SlBu:1992}
I.~H. Sloan and B.~J. Burn.
\newblock An unconventional quadrature method for logarithmic-kernel integral
  equations on closed curves.
\newblock {\em J. Integral Equations Appl.}, 4(1):117--151, 1992.

\end{thebibliography}

\end{document}